\documentclass{amsart}

\title[Categorical large cardinals]{Categorical large cardinals and the tension between categoricity and set-theoretic reflection}

\author{Joel David Hamkins}
\address[Joel David Hamkins]
{O'Hara Professor of Philosophy and Mathematics, University of Notre Dame, 100 Malloy Hall, Notre Dame, IN 46556 USA}
\email{jdhamkins@nd.edu}
\urladdr{http://jdh.hamkins.org}

\author{Hans Robin Solberg}
\address[Hans Robin Solberg]
{Doctoral student, Philosophy, University of Oxford}
\email{robin.solberg@philosophy.ox.ac.uk}
\urladdr{http://users.ox.ac.uk/~sedm5950/index.html}

\thanks{Commentary can be made about this article on the first author's blog at \href{http://jdh.hamkins.org/categorical-large-cardinals}{http://jdh.hamkins.org/categorical-large-cardinals}.}
\usepackage{latexsym,amsfonts,amsmath,amssymb,mathrsfs}
\usepackage{bbm}
\usepackage[hidelinks]{hyperref}
\usepackage{relsize}
\usepackage[rgb,dvipsnames]{xcolor}
\usepackage{tikz}
\usepackage{pgflibrarysnakes}
\usetikzlibrary{arrows,arrows.meta,petri,topaths,positioning,shapes,shapes.misc,patterns,calc,decorations.pathreplacing,hobby,snakes}
\usepackage{wrapfig} % for figures with wrapping text
\usepackage{float}
\usepackage[utf8]{inputenc}
\RequirePackage{doi}

\usepackage{enumitem}
\newtheorem{theorem}{Theorem}
\newtheorem*{theorem*}{Theorem}

\newtheorem*{maintheorem*}{Main Theorem}
\newtheorem*{maintheorems*}{Main Theorems}

\newtheorem*{corollary*}{Corollary}
\newtheorem*{corollaries*}{Corollaries}

\newtheorem{observation}[theorem]{Observation}

\theoremstyle{definition}

\newtheorem*{definition*}{Definition}
\newtheorem{maindefinition}[theorem]{Main Definition}

\newtheorem{question}[theorem]{Question}
\newtheorem*{question*}{Question}

\newtheorem*{questions*}{Questions}
\newtheorem{mainquestion}[theorem]{Main Question} % with numbering
\newtheorem*{mainquestion*}{Main Question} % without numbering
\newtheorem*{openquestion*}{Open Question} % without numbering
\theoremstyle{remark}

\newcommand{\QED}{\end{proof}}

\def\proclaim[#1]{{\bf #1}}
\def\BF#1.{{\bf #1.}}

\def\says#1:#2\par{\item[#1] #2\par}
\newcommand{\Godel}{G\"odel}

\newcommand{\Levy}{L\'{e}vy}
\newcommand{\Lowenheim}{L\"owenheim}

\newcommand{\C}{{\mathbb C}}

\newcommand{\N}{{\mathbb N}}

\newcommand{\R}{{\mathbb R}}

\newcommand{\continuum}{\mathfrak{c}}

\makeatletter
\newcommand{\dotminus}{\mathbin{\text{\@dotminus}}}
\newcommand{\@dotminus}{%
  \ooalign{\hidewidth\raise1ex\hbox{.}\hidewidth\cr$\m@th-$\cr}%
}
\makeatother

\newcommand{\of}{\subseteq}

\newcommand{\elesub}{\prec}

\newcommand{\Add}{\mathop{\rm Add}}

\newcommand{\Coll}{\mathop{\rm Coll}}

\newcommand{\Th}{\mathop{\rm Th}}
\newcommand{\Con}{\mathop{{\rm Con}}}

\newcommand{\satisfies}{\models}

\newcommand{\proves}{\vdash}
\renewcommand{\setminus}{\raise.3ex\hbox{\rotatebox{-20}{$-$}}} % the usual setminus is absurdly huge and vertical

\newcommand{\Union}{\bigcup}

\newcommand{\smalllt}{\mathrel{\mathchoice{\raise2pt\hbox{$\scriptstyle<$}}{\raise1pt\hbox{$\scriptstyle<$}}{\raise0pt\hbox{$\scriptscriptstyle<$}}{\scriptscriptstyle<}}}
\newcommand{\smallleq}{\mathrel{\mathchoice{\raise2pt\hbox{$\scriptstyle\leq$}}{\raise1pt\hbox{$\scriptstyle\leq$}}{\raise1pt\hbox{$\scriptscriptstyle\leq$}}{\scriptscriptstyle\leq}}}

   \def\DHLhksqrt#1#2{%
   \setbox0=\hbox{$#1\sqrt{#2\,}$}\dimen0=\ht0
   \advance\dimen0-0.2\ht0
   \setbox2=\hbox{\vrule height\ht0 depth -\dimen0}%
   {\box0\lower0.4pt\box2}}
\newcommand{\boolval}[1]{\mathopen{\lbrack\!\lbrack}\,#1\,\mathclose{\rbrack\!\rbrack}}
\def\[#1]{\boolval{#1}}
\newbox\gnBoxA
\newbox\gnBoxB
\newdimen\gnCornerHgt
\setbox\gnBoxA=\hbox{\tiny$\ulcorner$}
\global\gnCornerHgt=\ht\gnBoxA
\newdimen\gnArgHgt
\def\gcode #1{%
\setbox\gnBoxA=\hbox{$#1$}%
\setbox\gnBoxB=\hbox{$\bar #1$}%
\gnArgHgt=\ht\gnBoxB%
\ifnum     \gnArgHgt<\gnCornerHgt \gnArgHgt=0pt%
\else \advance \gnArgHgt by -\gnCornerHgt%
\fi \raise\gnArgHgt\hbox{\tiny$\ulcorner$} \box\gnBoxA %
\raise\gnArgHgt\hbox{\tiny$\urcorner$}}
\newcommand{\UnderTilde}[1]{{\setbox1=\hbox{$#1$}\baselineskip=0pt\vtop{\hbox{$#1$}\hbox to\wd1{\hfil$\sim$\hfil}}}{}}
\newcommand{\Undertilde}[1]{{\setbox1=\hbox{$#1$}\baselineskip=0pt\vtop{\hbox{$#1$}\hbox to\wd1{\hfil$\scriptstyle\sim$\hfil}}}{}}
\newcommand{\undertilde}[1]{{\setbox1=\hbox{$#1$}\baselineskip=0pt\vtop{\hbox{$#1$}\hbox to\wd1{\hfil$\scriptscriptstyle\sim$\hfil}}}{}}
\newcommand{\UnderdTilde}[1]{{\setbox1=\hbox{$#1$}\baselineskip=0pt\vtop{\hbox{$#1$}\hbox to\wd1{\hfil$\approx$\hfil}}}{}}
\newcommand{\Underdtilde}[1]{{\setbox1=\hbox{$#1$}\baselineskip=0pt\vtop{\hbox{$#1$}\hbox to\wd1{\hfil\scriptsize$\approx$\hfil}}}{}}

\renewcommand{\implies}{\mathrel{\rightarrow}}

\renewcommand{\iff}{\mathrel{\leftrightarrow}}

\newcommand{\iso}{\cong}
\def\<#1>{\left\langle#1\right\rangle}

\newcommand{\Ord}{\mathord{{\rm Ord}}}
\newcommand{\ZFC}{{\rm ZFC}}
\newcommand{\ZF}{{\rm ZF}}

\newcommand{\CH}{{\rm CH}}

\newcommand{\GCH}{{\rm GCH}}

\newcommand{\PA}{{\rm PA}}

\newcommand{\cell}[1]{\boxit{\hbox to 17pt{\strut\hfil$#1$\hfil}}}
\newcommand{\head}[2]{\lower2pt\vbox{\hbox{\strut\footnotesize\it\hskip3pt#2}\boxit{\cell#1}}}
\newcommand{\boxit}[1]{\setbox4=\hbox{\kern2pt#1\kern2pt}\hbox{\vrule\vbox{\hrule\kern2pt\box4\kern2pt\hrule}\vrule}}
\newcommand{\Col}[3]{\hbox{\vbox{\baselineskip=0pt\parskip=0pt\cell#1\cell#2\cell#3}}}
\newcommand{\tapenames}{\raise 5pt\vbox to .7in{\hbox to .8in{\it\hfill input: \strut}\vfill\hbox to
.8in{\it\hfill scratch: \strut}\vfill\hbox to .8in{\it\hfill output: \strut}}}
\newcommand{\Head}[4]{\lower2pt\vbox{\hbox to25pt{\strut\footnotesize\it\hfill#4\hfill}\boxit{\Col#1#2#3}}}
\newcommand{\Dots}{\raise 5pt\vbox to .7in{\hbox{\ $\cdots$\strut}\vfill\hbox{\ $\cdots$\strut}\vfill\hbox{\
$\cdots$\strut}}}
\usepackage[backend=bibtex,style=alphabetic,maxbibnames=15,maxcitenames=6,dateabbrev=false]{biblatex}
\renewcommand{\UrlFont}{} % makes url text smaller (used only in bibliography?)
\renewbibmacro{in:}{\ifentrytype{article}{}{\printtext{\bibstring{in}\intitlepunct}}} % removes "In:" in article format.
\DeclareFieldFormat{url}{{\UrlFont\url{#1}}} % removes redundant "URL:" text, and just gives the url.
\DeclareFieldFormat{urldate}{% improves setting of urldate
  (version \thefield{urlday}\addspace%
  \mkbibmonth{\thefield{urlmonth}}\addspace%
  \thefield{urlyear}\isdot)}
\DeclareFieldFormat{eprint:arxiv}{%   JDH format for arxiv info in bibliography
\ifhyperref
    {\href{http://arxiv.org/abs/#1}{%
    arXiv\addcolon\nolinkurl{#1}}\iffieldundef{eprintclass}{}{\UrlFont{\mkbibbrackets{\thefield{eprintclass}}}}}
    {arXiv\addcolon\nolinkurl{#1}\iffieldundef{eprintclass}{}{\UrlFont{\mkbibbrackets{\thefield{eprintclass}}}}}}
    \bibliography{HamkinsBiblio,MathBiblio,PhilBiblio,WebPosts}
\renewcommand\emptyset{\varnothing}

\newcommand{\leqequiv}{\mathrel{\text{\tikz[scale=1.3ex/1cm,baseline=-.9ex,line width=.1ex]{\draw (0,0) -- (1,.4) (0,0) -- (1,-.4) (0,-.3) -- (1,-.7) (0,-.6) -- (1,-1);}}}}
\newcommand{\Plus}{{\mathop{\text{\tikz[scale=.8ex/1cm,baseline=0ex,line width=.33ex]{\draw (-1,0) -- (1,0) (0,-1) -- (0,1);}}}}}

\begin{document}

\begin{abstract}
Inspired by Zermelo's quasi-categoricity result characterizing the models of second-order Zermelo-Fraenkel set theory $\ZFC_2$, we investigate when those models are fully categorical, characterized by the addition to $\ZFC_2$ either of a first-order sentence, a first-order theory, a second-order sentence or a second-order theory. Thus we mount an analysis of the \emph{categorical} large cardinals. This mathematical analysis leads naturally to philosophical issues concerning structuralism and realism, including especially the tension between categoricity and reflection. Ultimately we identify grounds for the preference of noncategoricity in one's foundations.
\end{abstract}

\maketitle

Categorical accounts of various mathematical structures lie at the very core of structuralist mathematical practice, enabling mathematicians to refer to specific mathematical structures, not by having carefully to prepare and point at specially constructed instances---preserved like the one-meter platinum bar locked in a case in Paris---but instead merely by mentioning features that uniquely characterize the structure up to isomorphism.

The natural numbers $\<\N,0,S>$, for example, are uniquely characterized by the Dedekind axioms, which assert that $0$ is not a successor, that the successor function $S$ is one-to-one, and that every set containing $0$ and closed under successor contains every number \cite{Dedekind1888:What-are-numbers-and-what-should-they-be,Dedekind:Essays-on-the-theory-of-numbers}. We know what we mean by the natural numbers---they have a definiteness---because we can describe features that completely determine the natural number structure. The real numbers $\<\R,+,\cdot,0,1>$ similarly are characterized up to isomorphism as the unique complete ordered field \cite{Huntington1903:Complete-sets-of-postulates-for-the-theory-of-real-quantities}. The complex numbers $\<\C,+,\cdot>$ form the unique algebraically closed field of characteristic $0$ and size continuum, or alternatively, the unique algebraic closure of the real numbers. Essentially all our fundamental mathematical structures enjoy such categorical characterizations, where a theory is \emph{categorical} if it identifies a unique mathematical structure up to isomorphism---any two models of the theory are isomorphic. In light of the \Lowenheim-Skolem theorem, which prevents categoricity for infinite structures in first-order logic, these categorical theories are generally made in second-order logic.\goodbreak

\newpage
In set theory, Zermelo characterized the models of second-order Zermelo-Fraenkel set theory $\ZFC_2$ in his famous quasi-categoricity result:

\begin{theorem}[Zermelo \cite{Zermelo1930:UberGrenzzahlenUndMengenBereiche}\footnote{Zermelo was more generally concerned with the urelement-based versions of set theory, and what he proved in \cite{Zermelo1930:UberGrenzzahlenUndMengenBereiche} is that the models of second-order $\ZF-$infinity with urelements are determined by two cardinals---their inaccessible cardinal height (allowing $\omega$ when the infinity axiom fails) and the number of urelements. In this article, we shall focus on the case of $\ZFC_2$, where there are no urelements and the axiom of infinity holds.}]
\label{Theorem.Zermelo-quasi-categoricity}
The models of $\ZFC_2$ are precisely those isomorphic to a rank-initial segment $\<V_\kappa,\in>$ of the cumulative set-theoretic universe $V$ cut off at an inaccessible cardinal $\kappa$. In particular, for any two models of $\ZFC_2$, one of them is isomorphic to a rank-initial segment of the other.
\end{theorem}

To prove this, Zermelo observed that if $M$ is a model of the second-order axiomatization of set theory $\ZFC_2$, with the full second-order replacement axiom, then $M$ will be well-founded, since it will contain all its $\omega$-sequences; so it will be (isomorphic to) a transitive set; it will be correct about power sets; consequently, $M$'s internal cumulative $V_\alpha$ hierarchy will agree with the actual $V_\alpha$ hierarchy, and the height $\kappa=\Ord^M$ will have to be both regular and a strong limit. So $M$ will be $V_\kappa$ for some inaccessible cardinal; and conversely, all such $V_\kappa$ are models of $\ZFC_2$. It follows that for any two models of $\ZFC_2$, one of them is isomorphic to an initial segment of the other.

Zermelo had explicitly considered these structures $\<V_\kappa,\in>$ as a set-theoretic universe concept, and we accordingly refer to them as \emph{Zermelo-Grothendieck universes}, in light both of Zermelo's 1930 analysis and Grothendieck's subsequent rediscovery and use of them in category theory; they feature in the \emph{universe axiom}, which asserts that every set is an element of some such $V_\kappa$, or equivalently, that there are unboundedly many inaccessible cardinals.

In this article, we seek to investigate the extent to which Zermelo's quasi-categoricity analysis can rise fully to the level of categoricity, in light of the observation that many of the $V_\kappa$ universes are categorically characterized by their sentences or theories.

\begin{mainquestion}
Which models of $\ZFC_2$ satisfy fully categorical theories?
\end{mainquestion}

\noindent If $\kappa$ is the smallest inaccessible cardinal, for example, then up to isomorphism $\<V_\kappa,\in>$ is the unique model of $\ZFC_2$ satisfying the first-order sentence ``there are no inaccessible cardinals.'' The least inaccessible cardinal is therefore an instance of what we call a first-order \emph{sententially categorical} cardinal. Similar ideas apply to the next inaccessible cardinal, and the next, and so on for quite a long way. Many of the inaccessible universes thus satisfy categorical theories extending $\ZFC_2$ by a sentence or theory, either in first or second order, and we should like to investigate these categorical extensions of $\ZFC_2$.

In addition, we shall discuss the philosophical relevance of categoricity and point particularly to the philosophical problem posed by the tension between the widespread support for categoricity in our fundamental mathematical structures with set-theoretic ideas on reflection principles, which are at heart anti-categorical. Ultimately we shall identify grounds for preferring noncategoricity in one's foundational theory.\goodbreak

\newpage
\section{Main definition and preliminary results}

Our main theme concerns several diverse notions of categoricity in set theory.

\begin{maindefinition}\upshape\
 \begin{enumerate}
   \item A cardinal $\kappa$ is \emph{first-order sententially categorical}, if there is a first-order sentence $\sigma$ in the language of set theory, such that $V_\kappa$ is categorically characterized by $\ZFC_2+\sigma$.
   \item A cardinal $\kappa$ is \emph{first-order theory categorical}, if there is a first-order theory $T$ in the language of set theory, such that $V_\kappa$ is categorically characterized by $\ZFC_2+T$.
   \item A cardinal $\kappa$ is \emph{second-order sententially categorical}, if there is a second-order sentence $\sigma$ in the language of set theory, such that $V_\kappa$ is categorically characterized by $\ZFC_2+\sigma$.
   \item A cardinal $\kappa$ is \emph{second-order theory categorical}, if there is a second-order theory $T$ in the language of set theory, such that $V_\kappa$ is categorically characterized by $\ZFC_2+T$.
 \end{enumerate}
\end{maindefinition}

One may easily refine and extend these definitions, of course, by stratifying on complexity. Thus, we will have natural notions of $\Sigma^m_n$ categoricity in $m$th-order set theory, including $\Sigma^\alpha_n$ categoricity for transfinite order $\alpha$, for either theories or sentences---these amount to first-order assertions in $V_{\kappa+\alpha}$. And one may also consider categoricity in infinitary languages and so on. In this article, however, let us focus principally on the categoricity notions in the main definition.

This definition can be seen as an instance of notions considered by Stephen Garland in \cite{Garland1967:Second-order-cardinal-characterizability} and also in his dissertation \cite{Garland1967:Dissertation}, in which he considers the ordinals $\kappa$ that are characterizable in pure second-order logic, that is, using the pure set $\kappa$ itself, rather than $\<V_\kappa,\in>$, and concentrating especially on $\Delta^1_2$ characterizations. Such a perspective is also undertaken in \cite{Väänänen2012:Second-order-logic-or-set-theory}.

We should like to emphasize, however, that because we aim at fully categorical refinements of Zermelo's theorem, we are focused here on categoricity specifically for extensions of the second-order theory $\ZFC_2$. In particular, what we call first-order sententially and first-order theory categorical amounts ultimately to the categoricity of a second-order sentence $\ZFC_2+\sigma$ or second-order theory $\ZFC_2+T$, respectively, even when $\sigma$ and $T$ are first-order. In effect, we grade the complexity of a categorical account in our main definition by the logical resources it requires beyond $\ZFC_2$, as to whether it is a first-order extension or second-order and whether it is a sentence or theory.\footnote{The main distinctions of our main definition, between first-order and second-order extensions of $\ZFC_2$ and between sentential and theory extensions, do not seem to appear in \cite{Garland1967:Second-order-cardinal-characterizability} or \cite{Väänänen2012:Second-order-logic-or-set-theory}.}

Since Zermelo characterized the inaccessible cardinals $\kappa$ as those for which $V_\kappa\satisfies\ZFC_2$, all the cardinals $\kappa$ in the main definition above are inaccessible. We could equivalently have defined that $\kappa$ is \emph{first-order sententially categorical} if there is a first-order sentence $\sigma$ such that $\kappa$ is the only inaccessible cardinal for which $V_\kappa\satisfies\sigma$. And similarly with the other kinds of categorical cardinals. In this sense, the topic is about categorical characterizations of $V_\kappa$ for inaccessible $\kappa$.

We find it interesting to notice that theory categoricity is akin to Leibnizian discernibility, since $\kappa$ is theory categorical precisely when $V_\kappa$ can be distinguished from other candidates $V_\lambda$ by a sentence, since that is exactly what it means for their theories to be different.

If there are any inaccessible cardinals at all, then there will be easy examples of categorical cardinals. As we mentioned earlier, the least inaccessible cardinal $\kappa$ is characterized over $\ZFC_2$ by ``there are no inaccessible cardinals.'' The next inaccessible cardinal is characterized by the first-order sentence, ``there is exactly one inaccessible cardinal.'' More generally, the $\alpha$th inaccessible cardinal (if we index from $0$) is characterized by the assertion ``there are exactly $\alpha$ many inaccessible cardinals.'' If the ordinal $\alpha$ is \emph{absolutely definable}, meaning that it is definable in a manner that is absolute to every inaccessible $V_\kappa$, then this assertion can be made without parameters and so the $\alpha$th inaccessible cardinal will be sententially categorical when it exists. So the categorical cardinals proceed from the beginning for quite a long way, up to the $\omega_1^{\textsc{ck}}$th inaccessible cardinal and beyond. Since we have observed that the smallest large cardinals are generally categorical, this is a sense in which categoricity is a smallness notion amongst the large cardinals, while largeness is noncategorical. Yet, the $\omega_1$th inaccessible cardinal if it exists is sententially categorical, and the $\omega_2$nd, and more. And thus the door is opened to the possibility of gaps in the categorical cardinals, since there are only countably many sentences.

\begin{observation}
 Categoricity is downward absolute from $V$ to any $V_\theta$. That is, if $\kappa$ is categorical in one of the four manners of the main definition and $\theta>\kappa$, then the structure $V_\theta$ knows that $\kappa$ is categorical in that way.
\end{observation}

\begin{proof}
The point is that $V_\theta$ can verify that $V_\kappa$ has whatever theory it has and there are fewer challenges to categoricity in $V_\theta$ than in $V$. Every inaccessible cardinal $\delta$ in $V_\theta$ is also inaccessible in $V$, and consequently differs in its theory from $V_\kappa$, and $V_\theta$ can see this.
\end{proof}

Let us consider now upward absoluteness, which might seem at first to be too much to ask for, since perhaps a cardinal $\kappa$ can be categorical inside $V_\theta$ only because $\theta$ is not large enough to reveal the other models $V_\lambda$ that satisfy that same characterization. But for sentential categoricity, it turns out that this situation never actually arises, and consequently sentential categoricity is fully absolute between $V$ and any $V_\theta$.

\begin{theorem}\label{Theorem.Sentential-categoricity-absolute-to-Vtheta}
Sentential categoricity (first and second order) is absolute between $V$ and any $V_\theta$, both upward and downward. That is, if $\kappa<\theta$, then $\kappa$ is sententially categorical in $V$ if and only if it has the same sentential categoricity in $V_\theta$.
\end{theorem}

\begin{proof}
If $\kappa$ is sententially categorical in $V$, then it is sententially categorical in $V_\theta$ by the same sentence, since there are fewer competitors inside $V_\theta$ than in $V$.

Conversely, suppose that $\kappa$ is sententially categorical inside $V_\theta$ with $\kappa<\theta$. So there is a sentence $\sigma$, such that $V_\kappa\satisfies\sigma$ and no other $V_\lambda$ in $V_\theta$ satisfies $\sigma$. In particular, $V_\kappa$ is the first inaccessible level to satisfy $\sigma$, and so the sentence $\sigma+$``there is no inaccessible $\delta$ with $V_\delta\satisfies\sigma$'' is a categorical characterization of $V_\kappa$ in the full universe $V$. No larger inaccessible level can satisfy this sentence, since $V_\kappa\satisfies\sigma$. So $\kappa$ is sententially categorical in $V$.
\end{proof}

Another way to describe the upward absoluteness is that failures of sentential categoricity are always witnessed by \emph{smaller} as opposed to larger inaccessible cardinals with the same sentence. Because of this, every failure of sentential categoricity provides instances of inaccessible reflection. In the first-order case, one can view this as a weak form of Mahloness, since every Mahlo cardinal has a robust first-order inaccessible reflection property. In the second-order case, it is a weak form of indescribability.

\begin{theorem}\label{Theorem.Failures-of-sentential-categoricity-smaller}
If an inaccessible cardinal $\kappa$ is not sententially categorical (either first or second order), then every sentence $\sigma$ of that order that is true in $V_\kappa$ is also true in $V_\delta$ for some smaller inaccessible cardinal $\delta<\kappa$.
\end{theorem}

\begin{proof}
If $\sigma$ is true in $V_\kappa$, then it cannot be the first time this happens at an inaccessible cardinal, since otherwise this very situation could be described in a sentence, providing a categorical characterization of $\kappa$.
\end{proof}

The corresponding fact is not true for theory categoricity, if one expects to reflect the whole theory, because failures of theory categoricity arise when there are inaccessible $\kappa<\lambda$ for which $V_\kappa$ and $V_\lambda$ have the same theory. But in this case, there will be a smallest such $\kappa$ with that particular theory, and this $\kappa$ is not theory categorical, but this is not witnessed by any smaller $\delta<\kappa$ precisely because $\kappa$ was already the smallest with that theory. In particular, the analogue of theorem \ref{Theorem.Sentential-categoricity-absolute-to-Vtheta} also fails for theory categoricity, because if $\kappa$ is the least inaccessible cardinal that is not theory categorical, then there will be some $\lambda>\kappa$ with the same theory, and if $\lambda$ is smallest with this property, then $\kappa$ will be theory categorical inside $V_\lambda$, but not in $V$. Nevertheless, there is an approximation version of downward reflection for failures of theory categoricity, which can be seen as a strengthening of weak Mahloness.\goodbreak

\begin{theorem}\label{Theorem.Non-theory-categorical-cardinals-are-reflecting}
If an inaccessible cardinal $\kappa$ is not first-order theory categorical, then for every natural number $n$, there is a smaller inaccessible cardinal $\delta<\kappa$ for which $V_\delta$ has the same $\Sigma_n$ theory. And similarly, if $\kappa$ is not second-order theory categorical, there for every $n$ there is a smaller inaccessible $V_\delta$ with the same $\Sigma^1_n$ theory.
\end{theorem}

\begin{proof}
Suppose that $\kappa$ is not first-order theory categorical. Then there is some other inaccessible cardinal $\lambda$ for which $V_\kappa$ and $V_\lambda$ have the same theory. If $\lambda<\kappa$, then we are done immediately. So assume $\lambda>\kappa$. Notice that $V_\lambda$ thinks that its $\Sigma_n$ theory is shared by $V_\kappa$. So it is part of the theory of $V_\lambda$ that there is a smaller inaccessible cardinal with the same $\Sigma_n$ theory. So this statement is also true in $V_\kappa$, as desired. The argument works the same with first-order or second-order.
\end{proof}

The general phenomenon is that failures of categoricity lead to inaccessible reflection, which we view as weak forms of Mahloness and indescribability.

Let us next observe that in principle any statement or theory that is consistent with $\ZFC_2$ can be categorical.

\begin{theorem}\label{Theorem.Every-theory-can-be-categorical}
 Any sentence or theory (whether first or second order) that is true in some set-theoretic universe $V_\kappa\models\ZFC_2$ is either a categorical characterization of this model $V_\kappa$ over $\ZFC_2$ or else it is consistent with $\ZFC_2$ that it is such a categorical characterization of such a universe.
\end{theorem}

\begin{proof}
Suppose that $V_\kappa\models\ZFC_2+\Sigma$. If $\kappa$ is the only inaccessible cardinal for which $V_\kappa$ satisfies $\Sigma$, then this is a categorical characterization of $V_\kappa$ and we are done. Otherwise, there are at least two such inaccessible cardinals. If $\kappa$ is the least for which $V_\kappa\satisfies\ZFC_2+\Sigma$ and $\lambda$ is the next one, then $V_\lambda$ is a model of $\ZFC_2$ inside which only $V_\kappa$ satisfies $\ZFC_2+\Sigma$. So inside $V_\lambda$, this is a categorical characterization of $V_\kappa$.
\end{proof}\goodbreak

Thus, every theory that can be true in a set-theoretic universe can also be a categorical characterization of such a set-theoretic universe. In this sense, the answer to the question ``Which extensions of $\ZFC_2$ can be categorical?'' is that these are exactly the same extensions that can be true at all. The question whether a sentence or theory is categorical depends of course on what else there is, and in this sense, categoricity assertions amount to large cardinal non-existence assertions---one is asserting that there is no other inaccessible cardinal satisfying that same sentence or theory.

Isn't it mildly paradoxical for a model $V_\lambda$ to satisfy a theory $T$ and simultaneously to satisfy the statement that this theory $T$ categorically characterizes a different, strictly smaller structure $V_\kappa$, for $\kappa<\lambda$. This isn't contradictory, since the situation is that there there is only one such $V_\kappa$ inside $V_\lambda$, and so $V_\lambda$ looks upon the theory as categorical, even though it satisfies the same theory. The situation is rather similar to how a model of $\PA+\neg\Con(\PA)$ is mistaken about consistency, because it thinks that $\PA$ is true yet also inconsistent; our model $V_\lambda$ thinks the theory $T$ has no other models than $V_\kappa$, even though it itself is a second, distinct model.\goodbreak

The phenomenon occurs already with $\ZFC_2$ itself in the second inaccessible cardinal. Namely, if $\kappa_0<\kappa_1$ are the first two inaccessible cardinals, then $V_{\kappa_1}\satisfies\ZFC_2+``\ZFC_2$ is a categorical characterization of $V_{\kappa_0}$.'' The model $V_{\kappa_1}$ can thus be seen as simply mistaken about categoricity; and theorem \ref{Theorem.Every-theory-can-be-categorical} establishes the phenomenon generally with any theory. In this sense, perhaps we don't really want to know which extensions of $\ZFC_2$ \emph{can} be categorical, as seen in some (perhaps misguided) model, but rather which extensions \emph{are} categorical, in the fully complete set-theoretic universe $V$, not artificially truncated at some inaccessible cardinal $\delta$. Whether such a perspective is meaningful is, of course, a central question in the philosophy of set theory. Meanwhile, let us proceed to investigate the nature of categoricity mathematically, as a property of cardinals in \ZFC\ under diverse assumptions about what cardinals there might be. We shall return to the philosophical issue concerning the interaction of object theory and meta theory in the final section.

Categoricity exhibits a certain upward transmission effect from any categorical cardinal to larger instances of categoricity. For any cardinal $\kappa$, let us use the boldface successor notation $\kappa^\Plus$ to denote the next inaccessible cardinal above $\kappa$.

\begin{theorem}\label{Theorem.Kappa^Plus-sententially-categorical}
If $\kappa$ is first-order sententially categorical, or indeed merely second-order sententially categorical, or merely $\alpha$-order sententially categorical for some absolutely definable ordinal $\alpha$, then the next inaccessible cardinal $\kappa^\Plus$ is first-order sententially categorical.
\end{theorem}

\begin{proof}
If $\psi$ is the sentence characterizing $V_\kappa$, then the next inaccessible cardinal $V_{\kappa^\Plus}$ can see that there is a largest inaccessible cardinal $\kappa$ and that $V_\kappa$ satisfies $\psi$, and this property characterizes $\kappa^\Plus$.
\end{proof}

Notice that the sentence characterizing $\kappa^\Plus$ in theorem \ref{Theorem.Kappa^Plus-sententially-categorical} has complexity $\Delta_3$, since it is the conjunction of a $\Sigma_2$ sentence, asserting that there is an inaccessible cardinal $\kappa$ for which $V_\kappa\satisfies\psi$, and a $\Pi_2$ sentence, asserting that there are no inaccessible cardinals above such a $\kappa$.

This argument has an analogue with the theory categorical cardinals and also with what we call the fresh cardinals. Namely, an inaccessible cardinal $\kappa$ is \emph{fresh}, for its first or second order theory, as specified, if the corresponding theory of $V_\kappa$ does not arise as the theory of $V_\delta$ for any inaccessible $\delta<\kappa$. In other words, $\kappa$ is fresh when the theory of $V_\kappa$ occurs for the first time at $\kappa$. Every theory categorical cardinal, of course, is fresh, but this is not an equivalence when there are inaccessible cardinals $\kappa$ that are not theory categorical, since the first instance of any particular theory at an inaccessible cardinal will be an instance of freshness, even if that theory occurs again later, as it must occur when the cardinal is not theory categorical.

\begin{theorem}\label{Theorem.If-kappa-fresh-then-kappa-Plus-theory-categorical}
If $\kappa$ is fresh (either first-order, second-order, or indeed $\alpha$-order for any absolutely definable ordinal $\alpha$), then the next inaccessible cardinal $\kappa^\Plus$ is first-order theory categorical.
\end{theorem}

\begin{proof}
Suppose that $\kappa$ is fresh, in that the theory of $V_\kappa$ arises for the first time at an inaccessible cardinal right at $\kappa$ itself. Since $\kappa$ is definable in $V_{\kappa^\Plus}$ as the largest inaccessible cardinal, and it's theory is as well (including not only its first and  second-order theories, but $\alpha$-order for any absolutely definable ordinal $\alpha$), and so the assertion that any particular  sentence $\sigma$ of that order holds in $V_\kappa$ is a first-order assertion in $V_{\kappa^\Plus}$. Thus, $V_{\kappa^\Plus}$ will be characterized by the theory $\ZFC_2$ plus the first-order assertions that there is a largest inaccessible cardinal, that it is fresh and that has the particular theory that it has. This is a first-order theory characterization of $\kappa^\Plus$, and so $\kappa^\Plus$ is first-order theory categorical, as desired.
\end{proof}

\begin{theorem}\label{Theorem.Limits-of-sententially-categorical-are-theory-categorical}
If $\kappa$ is inaccessible and the sententially categorical cardinals are unbounded in the inaccessible cardinals below $\kappa$, then $\kappa$ is first-order theory categorical.
\end{theorem}

\begin{proof}
Note that we are not assuming that $\kappa$ is a limit of inaccessible cardinals. The first-order theory of $V_\kappa$ includes the assertions that those smaller inaccessible cardinals satisfy the sentences that characterize them, whether those are first or second order. Therefore no smaller inaccessible cardinal $\delta<\kappa$ can have $V_\delta$ with the same theory as $V_\kappa$. And no larger $\theta>\kappa$ can have the same theory either, since in $V_\theta$ either there are new sententially categorical cardinals, and the assertion that they exist would be part of the theory of $V_\theta$ not true in $V_\kappa$, or else the sententially categorical cardinals will not be unbounded in the inaccessible cardinals below $\theta$, which will itself be a statement true in $V_\theta$ that is not true in $V_\kappa$.
\end{proof}

%It follows that if there are unboundedly many inaccessible cardinals, then the supremum of the sententially categorical cardinals, mentioned in theorem \ref{Theorem.Supremum-sententially-categorical-cardials}, is strictly less than the supremum of the theory categorical cardinals, mentioned in theorem \ref{Theorem.Supremum-theory-categorical-cardials}.
%
%Skip?
%\begin{corollary}\label{Corollary.Supremum-sentential-less-than-supremum-theory}
%If the class of inaccessible cardinals has uncountable cofinality (for example, if it forms a proper class), then the supremum of the sententially categorical cardinals, as in theorem \ref{Theorem.Supremum-sententially-categorical-cardials}, is strictly less than the supremum of the theory categorical cardinals, as in theorem \ref{Theorem.Supremum-theory-categorical-cardials}.
%\end{corollary}
%
%\begin{proof}
%Since there are only countably many first-order sententially categorical cardinals, there must be an inaccessible cardinal larger than every sententially categorical cardinal (of any absolutely definable order). By theorem \ref{Theorem.Limits-of-sententially-categorical-are-theory-categorical}, the least such $\kappa$ will be theory categorical. So the supremum of the theory categorical cardinals strictly exceeds the supremum of the sententially categorical cardinals.
%\end{proof}

\begin{theorem}\label{Theorem.Mahlo-not-first-order-theory-categorical}
No Mahlo cardinal is first-order theory categorical.
\end{theorem}

\begin{proof}
If $\kappa$ is Mahlo, then $V_\delta\elesub V_\kappa$ for a closed unbounded set of $\delta$, which therefore includes many inaccessible cardinals. So $V_\kappa$ is not characterized by any first-order sentence or theory.
\end{proof}

\begin{theorem}\label{Theorem.Least-Mahlo}
The least Mahlo cardinal is second-order sententially categorical, but not first-order theory categorical.
\end{theorem}

\begin{proof}
Being Mahlo is a $\Pi^1_1$ property: every club $C\of\kappa$ has a regular cardinal. So the least one is second-order sententially categorical. But no Mahlo cardinal is first-order sententially or theory categorical by the above.
\end{proof}

We may consider versions of theorem \ref{Theorem.Mahlo-not-first-order-theory-categorical} for weaker large cardinal notions. An ordinal $\kappa$ is \emph{otherworldly} if $V_\kappa\elesub V_\beta$ for some ordinal $\beta>\kappa$. (It is an interesting elementary exercise to show that otherworldly cardinals are also \emph{worldly}, which means $V_\kappa\satisfies\ZFC$, but note that otherworldly cardinals, like the worldly cardinals, need not be regular and hence are not necessarily inaccessible, although they are strong limit cardinals and $\beth$-fixed points.) Every inaccessible cardinal $\delta$ is the limit of a closed unbounded set of otherworldly cardinals, each of which is otherworldly up to $\delta$, meaning that the targets $\beta$ can be found cofinally in $\delta$. A cardinal $\kappa$ is \emph{totally otherworldly} if $V_\kappa\elesub V_\beta$ for arbitrarily large ordinals $\beta$. Recall from \cite{HamkinsJohnstone2014:ResurrectionAxiomsAndUpliftingCardinals} that a cardinal $\kappa$ is \emph{uplifting} if it is inaccessible and $V_\kappa\elesub V_\beta$ for arbitrarily large inaccessible cardinals $\beta$; the cardinal $\kappa$ is \emph{pseudo-uplifting} if $V_\kappa\elesub V_\beta$ for arbitrarily large $\beta$, without insisting that $\beta$ is inaccessible. So the pseudo-uplifting cardinals are the same as the inaccessible totally otherworldly cardinals. All of these cardinals are weaker in consistency strength than a Mahlo cardinal, since if $\delta$ is Mahlo, then $V_\delta$ has a proper class of uplifting cardinals, which are therefore also pseudo-uplifting and inaccessibly totally otherworldly. None of these types of cardinals, it turns out, can be second-order theory categorical, in light of the following theorem.\goodbreak

\begin{theorem}\label{Theorem.Extensions-not-second-order-theory-categorical}
If $V_\kappa\elesub M$ for some transitive set $M$ with $V_{\kappa+1}\of M$, then $\kappa$ is neither first nor second-order theory categorical.
\end{theorem}

\begin{proof}
Assume $V_\kappa\elesub M$ for a transitive set $M$ with $V_{\kappa+1}\of M$. Let $T$ be the second-order theory of $V_\kappa$, and observe that $M$ thinks ``there is an inaccessible cardinal $\delta$ such that the second-order theory of $V_\delta$ is $T$,'' since this is true in $M$ of $\delta=\kappa$. So this statement must also be true in $V_\kappa$, and so there is $\delta<\kappa$ with $V_\delta$ having the same theory as $V_\kappa$. So $\kappa$ is not second-order theory categorical.
\end{proof}

It follows that no measurable cardinal is second-order theory categorical, nor is any $(\kappa+1)$-strongly unfoldable cardinal $\kappa$, nor any $\Pi^2_1$-indescribable cardinal, nor any uplifting cardinal, nor any otherworldly cardinal, since all these types of large cardinals satisfy the comparatively weak hypothesis of theorem \ref{Theorem.Extensions-not-second-order-theory-categorical}.

\section{The limits of categoricity}

Let us now analyze the limits and supremum of the various kinds of categorical cardinals. We define that an ordinal $\delta$ is (lightface) \emph{$\Sigma_2$-correct}, if $V_\delta$ has the same $\Sigma_2$ theory as $V$. Generalizing this, $\delta$ is \emph{$\Sigma_2(\R)$-correct}, if $V_\delta$ satisfies the same $\Sigma_2$ assertions as $V$ with real parameters; and $\delta$ is (fully) \emph{$\undertilde\Sigma_2$}-correct, if $V_\delta\elesub_{\Sigma_2} V$, meaning that $V_\delta$ has the same $\Sigma_2$ theory as $V$ with arbitrary parameters from $V_\delta$.

The $\Sigma_2$ assertions are exactly those that can be locally verified in some rank-initial segment of the universe $V_\beta$. (The first author gives an elementary account in \cite{Hamkins.blog2014:Local-properties-in-set-theory}, which may be helpful for some readers.) The least lightface $\Sigma_2$-correct cardinal, therefore, is simply the supremum of the ordinals $\beta$ at which new $\Sigma_2$ facts first become true.

\begin{theorem}\label{Theorem.Correctness-relations}\
 \begin{enumerate}
   \item Every sententially categorical cardinal is less than every lightface $\Sigma_2$-correct cardinal.
   \item Every theory categorical cardinal and indeed any fresh cardinal is less than every $\Sigma_2(\R)$-correct cardinal.
   \item The least lightface $\Sigma_2$-correct cardinal is strictly less than the least $\Sigma_2(\R)$-correct cardinal, which is strictly less than the least $\undertilde\Sigma_2$-correct cardinal.
 \end{enumerate}
\end{theorem}

\begin{proof}
If $\kappa$ is sententially characterized by sentence $\sigma$, then the assertion that there is an inaccessible cardinal $\kappa$ for which $V_\kappa\satisfies\sigma$ is itself a $\Sigma_2$ assertion in $V$, and so every sententially categorical cardinal (in any absolutely definable order) is less than every lightface $\Sigma_2$-correct cardinal.

Similarly, if $\kappa$ is theory categorical (in some absolutely definable order $\alpha$), then $V_\kappa$ is characterized by having its theory, and so the assertion that there is an inaccessible cardinal at which that theory becomes true is a new $\Sigma_2$ statement about that theory that first becomes true above $\kappa$. Since the theory can be coded by a real number, $\kappa$ must be less than every $\Sigma_2(\R)$-correct cardinal. An essentially similar argument works with fresh cardinals.

If $\xi$ is the least lightface $\Sigma_2$-correct cardinal, then $\xi$ is the first stage at which all the $\Sigma_2$ sentences that will ever become true have already become true, so the $\Sigma_2$ theory of $V_\xi$ is the same as $V$. If $t$ is the $\Sigma_2$ theory of $V_\xi$, then $\xi$ is the first time that this theory is realized in a rank-initial segment of the universe, and this is a $\Sigma_2$ property about the parameter $t$. So $\xi$ will be strictly less than the least $\Sigma_2(\R)$-correct cardinal $\theta$.

This cardinal $\theta$ in turn is strictly less, we claim, than the first $\undertilde\Sigma_2$-correct cardinal $\delta$. To see this, assume $V_\delta\elesub_{\Sigma_2} V$ and consider the $\Sigma_2(\R)$ theory of $V$, the set of $\Sigma_2$ assertions $\varphi(r)$ with real parameters that are true in $V$. These are also all true in $V_\delta$, by assumption. In $V$ therefore there is an ordinal stage, namely $V_\delta$, at which this theory becomes true. This is a $\Sigma_2$ fact about this theory, which is a set in $V_\delta$, which therefore must also be already true in $V_\delta$. So $\delta$ is not the first ordinal with the same $\Sigma_2(\R)$ theory as $V$, as claimed.
\end{proof}

It is a routine observation of large cardinal set theory that many of the familiar large cardinals are $\undertilde\Sigma_2$-correct. For example, every strong cardinal, every totally otherworldly cardinal, every strong cardinal, every supercompact cardinal, every extendible cardinal, and so on. The conclusion to be made from theorem \ref{Theorem.Correctness-relations}, therefore, is that every sententially categorical and theory categorical cardinal is below all of these large cardinals, instantiating once again the idea that categoricity is a smallness notion. Note that $\undertilde\Sigma_2$-correctness is not by itself a large cardinal notion, since the reflection theorem shows in \ZFC\ that for every $n$ there is a proper class club of $\undertilde\Sigma_n$-correct cardinals. Rather, one should see the correct cardinals as natural milestones of reflection in the ordinals, which exist inside any model of \ZFC.

Next, we show that the bound of theorem \ref{Theorem.Correctness-relations} statement (1) is optimal, because if there are a proper class of inaccessible cardinals, then the supremum of the sententially categorical cardinals is exactly the least lightface $\Sigma_2$-correct cardinal.

\begin{theorem}\label{Theorem.Supremum-sententially-categorical-cardials}
If there is no largest inaccessible cardinal, then the following ordinals are the same.
\begin{enumerate}
  \item The supremum of the first-order sententially categorical cardinals.
  \item The supremum of the second-order sententially categorical cardinals.
  \item The supremum of the $\alpha$-order sententially categorical cardinals, for any absolutely definable ordinal $\alpha$.
\end{enumerate}
The ordinals of statements 1, 2, and 3 are in any case less than or equal to the following ordinal.
\begin{enumerate}[resume]
  \item The first lightface $\Sigma_2$-correct cardinal.
\end{enumerate}
If there are a proper class of inaccessible cardinals, then all four ordinals are identical.
\end{theorem}\goodbreak

\begin{proof}
Every first-order sententially categorical cardinal is of course also second-order and indeed $\alpha$-order sententially categorical, and conversely by theorem \ref{Theorem.Kappa^Plus-sententially-categorical} every such sententially categorical cardinal is smaller than some first-order sententially categorical cardinal, and so the suprema of these cardinals, as in statements 1, 2, and 3, are the same.

We have already proved in theorem \ref{Theorem.Correctness-relations} that the ordinals of statements 1, 2, and 3 are bounded above by the first lightface $\Sigma_2$-correct cardinal. To complete the proof, assume that there are a proper class of inaccessible cardinals. Whenever a new $\Sigma_2$ fact is verified in some $V_\beta$, therefore, then the first inaccessible cardinal $\kappa$ above $\beta$ is categorically characterized under this description, and so the supremum of the ordinals at which a new $\Sigma_2$ becomes true will be bounded by the supremum of the sententially categorical cardinals. In this case, therefore, all four suprema are identical.
\end{proof}

For the final conclusion, it would suffice that for every true $\Sigma_2$ sentence there is an inaccessible cardinal at or above the stage at which it first becomes verified, and this is a weaker assumption than a proper class of inaccessible cardinals.

One can interpret theorem \ref{Theorem.Supremum-sententially-categorical-cardials} as providing support for $\ZFC_2$ as a natural end-of-the-line with regard to characterizability power, since the first-order and second-order characterizations over $\ZFC_2$ are cofinal in each other.

Let us now mount the corresponding analysis for theory categorical cardinals.

\begin{theorem}\label{Theorem.Supremum-theory-categorical-cardials}
If there is no largest inaccessible cardinal, then the following ordinals are the same:
\begin{enumerate}
  \item The supremum of the first-order theory categorical cardinals.
  \item The supremum of the second-order theory categorical cardinals.
  \item The supremum of the $\alpha$-order theory categorical cardinals, for any absolutely definable ordinal $\alpha$.
  \item The supremum of the first-order fresh cardinals.
  \item The supremum of the second-order fresh cardinals.
  \item The supremum of the $\alpha$-order fresh cardinals, for any absolutely definable ordinal $\alpha$.
\end{enumerate}
Furthermore, if there are a proper class of inaccessible cardinals, then this supremum is strictly higher than the first lightface $\Sigma_2$-correct cardinal, yet bounded above by the first $\Sigma_2(\R)$-correct ordinal.
\end{theorem}

\begin{proof}
Every first-order theory categorical cardinal is second-order theory categorical and $\alpha$-order theory categorical, as well as first-order fresh, and these are all also second-order fresh and $\alpha$-order fresh. By theorem \ref{Theorem.If-kappa-fresh-then-kappa-Plus-theory-categorical}, every $\alpha$-order fresh cardinal $\kappa$ has the next inaccessible cardinal $\kappa^\Plus$ being first-order theory categorical. So all various types of categorical and fresh cardinals are all interwoven cofinally in each other, and therefore the supremum of each class is the same.

By theorem \ref{Theorem.Supremum-sententially-categorical-cardials}, if there are unboundedly many inaccessible cardinals, then the supremum of the first-order sententially categorical cardinals is the first lightface $\Sigma_2$-correct cardinal, and the next inaccessible after this will be theory categorical by theorem \ref{Theorem.Limits-of-sententially-categorical-are-theory-categorical}. So the supremum of the theory categorical cardinals is strictly larger than the supremum of the sententially categorical cardinals, which is the first lightface $\Sigma_2$-correct cardinal.
If $\kappa$ is fresh, then if $t$ is the theory of $V_\kappa$, the assertion that there is an inaccessible cardinal whose rank-initial segment has theory $t$ is a $\Sigma_2$ fact about $t$ that first becomes true just after $\kappa$. So $\kappa$ is less than the supremum of the ordinals $\xi$ at which a new $\Sigma_2$ fact becomes true in $V_\xi$ with a real parameter, and so the supremum is bounded by the first $\Sigma_2(\R)$-correct cardinal.
\end{proof}

It follows immediately from the theorem that every $\alpha$-order theory arising as the theory of $V_\gamma$ for an inaccessible cardinal $\gamma$ above the supremum ordinal mentioned in the theorem, for any absolutely definable ordinal $\alpha$, also arises as the theory of some $V_\delta$ for an inaccessible cardinal $\delta$ strictly below this supremum. That is, every inaccessible cardinal $\gamma$ above this supremum (which is not regular since it has cofinality at most $\continuum$) has a theory for $V_\gamma$ that arises for the first time at a fresh cardinal $\beta$ strictly below the supremum. In short, every theory arising at inaccessible cardinals above the supremum mentioned in theorem \ref{Theorem.Supremum-theory-categorical-cardials} has already been encountered below it.

In this sense, the truly large cardinals, those above the supremum of theorem \ref{Theorem.Supremum-theory-categorical-cardials}, are never fresh---their theories, at whatever absolutely definable order $\alpha$ one might consider, have been encountered before below this supremum. Once we get above this supremum, the theories of the various $V_\kappa$ are simply repeating themselves in various complicated patterns that have already been seen before.

To summarize the consequences of theorems \ref{Theorem.Supremum-sententially-categorical-cardials} and \ref{Theorem.Supremum-theory-categorical-cardials} for the case when there is a proper class of inaccessible cardinals, the situation is that the supremum of the sententially categorical cardinals is equal to the first lightface $\Sigma_2$-correct cardinal, which is strictly less than the supremum of the theory categorical cardinals, which is bounded above by the least $\Sigma_2(\R)$-correct cardinal, which is strictly less than the least $\undertilde\Sigma_2$-correct cardinal.

The analysis of this section seems to justify the view of categoricity as a smallness notion for large cardinals---or perhaps it is better to say that the true largeness notions are anti-categorical---since categoricity occurs only below the $\undertilde\Sigma_2$-correct cardinals and therefore below every totally otherworldly cardinal, every strong cardinal, every supercompact cardinal, every extendible cardinal, and more, as these cardinals are all $\undertilde\Sigma_2$ correct. This view in turn leads to some philosophical tensions with the idea that mathematicians should generally seek categorical accounts of all their fundamental structures. After all, for us to adopt one of the categorical axiomatizations that characterize these categorical cardinals would be to adopt a theory that we know must be describing a smallish set-theoretic universe. But
this is the opposite of what we are trying to do in our set-theoretic foundations---we seek instead upward-reaching axioms that will maximize our foundational realm and make it as large as possible, so as more fully to accommodate arbitrary mathematical structure and ideas. We shall discuss this philosophical tension more fully in section \ref{Section.Philosophical-issues}. Meanwhile, the property of being $\undertilde\Sigma_2$ correct depends of course on the ambient set-theoretic universe $V$ in which it is considered.

Let us introduce the \emph{rank elementary forest} on inaccessible cardinals, the relation by which $\kappa\preceq\lambda$ if and only if the corresponding rank-initial segments form an elementary substructure, $V_\kappa\elesub V_\lambda$, for inaccessible cardinals $\kappa$ and $\lambda$. This is a partial order on inaccessible cardinals, and it is a forest, since the predecessors of any node are linearly ordered.

\begin{theorem}
Every first-order theory categorical cardinal is a stump in the rank elementary forest, that is, a disconnected root node with nothing above it.
\end{theorem}

\begin{proof}
This is almost immediate from the definition. If $\kappa$ is first-order theory categorical, then we cannot have $V_\kappa\elesub V_\lambda$ nor $V_\delta\elesub V_\kappa$, and so $\kappa$ must be a stump.
\end{proof}

The converse is not true, since we can have $V_\kappa\equiv V_\lambda$ without $V_\kappa\elesub V_\lambda$, and so it seems possible to violate first-order theory categoricity while $\kappa$ is still a stump. To see this, let $\kappa$ be the least inaccessible cardinal for which there is some ordinal $\lambda$ for which $V_\kappa\elesub V_\lambda$, and let $\lambda$ be least with that property. It follows that $V_\lambda$ thinks that the theory $T=\Th(V_\kappa)$ is realized at an inaccessible cardinal, namely at $\kappa$, and so $V_\kappa$ should also think this about $T$. So there will be some inaccessible cardinal $\delta<\kappa$ with $V_\delta\equiv V_\kappa$. So $V_\lambda$ will think that $\kappa$ is a stump in the rank elementary forest, but not first-order theory categorical.

Theory categoricity is ultimately about the relation of elementary equivalence $V_\kappa\equiv V_\lambda$ rather than the relation of elementary substructure $V_\kappa\elesub V_\lambda$, and so it is sensible to consider the alternative forest, by which $\kappa\leqequiv\lambda$ for inaccessible cardinals, just in case $\kappa\leq\lambda$ and $V_\kappa\equiv V_\lambda$. In this case, a cardinal is first-order theory categorical just in case it is a stump in the $\leqequiv$ forest. And there is a corresponding forest order $\leqequiv_2$ for second-order elementary equivalence $V_\kappa\equiv_2 V_\lambda$, by which the second-order theory categorical cardinals are exactly the stumps of the $\leqequiv_2$-forest.

\section{Gaps in the categorical cardinals}

Let us now prove that there are various kinds of gaps in the categorical cardinals. If there are sufficiently many inaccessible cardinals, then the categorical cardinals (of any type) do not form an initial segment of the inaccessible cardinals.

\begin{observation}
If there are uncountably many inaccessible cardinals, then there are some inaccessible cardinals that are neither first-order nor second-order sententially categorical.
\end{observation}

\begin{proof}
This is clear, because there are only countably many sentences, in either first or second order, and we may associate each sententially categorical cardinal with a sentence that characterizes it.\footnote{This is not an instance of the Math Tea argument, as discussed in \cite{HamkinsLinetskyReitz2013:PointwiseDefinableModelsOfSetTheory}, since we are not referring here to truth-in-the-universe $V$, but only to truth in set structures.} This association is one-to-one between the sententially categorical cardinals and a countable set. So there must be some inaccessible cardinals that are not sententially categorical.
\end{proof}\goodbreak

\begin{theorem}\
 \begin{enumerate}
   \item If $\kappa$ is the least inaccessible cardinal that is not first-order sententially categorical, then $\kappa^\Plus$, if it exists, is first-order sententially categorical.
   \item If $\kappa$ is the least inaccessible cardinal that is not second-order sententially categorical, then $\kappa^\Plus$, if it exists, is first-order sententially categorical.
   \item Consequently, if there are at least two inaccessible cardinals that are not sententially categorical, then there are gaps in the sententially categorical cardinals.
 \end{enumerate}
\end{theorem}

\begin{proof}
Suppose that $\kappa$ is the least inaccessible cardinal that is not first-order sententially categorical. By theorems \ref{Theorem.Sentential-categoricity-absolute-to-Vtheta} and \ref{Theorem.Failures-of-sentential-categoricity-smaller}, this is observable inside any larger $V_\theta$. In particular, if $\kappa^\Plus$ exists, then $V_{\kappa^\Plus}$ can see that $\kappa$ is not first-order sententially categorical. Therefore, $V_{\kappa^\Plus}$ is categorically characterized by the sentence asserting, ``there is a largest inaccessible cardinal and it is the only inaccessible cardinal that is not first-order sententially categorical.''

Essentially the same argument works if $\kappa$ is the least inaccessible cardinal that is not second-order sententially categorical. In this case, we still find a \emph{first-order} sentential characterization of $V_{\kappa^\Plus}$, since it will be the only inaccessible model which thinks there is a largest inaccessible cardinal which also is the only such cardinal that is not second-order sententially categorical, and the point is that this is a first-order assertion in $V_{\kappa^\Plus}$ since this model can define the truth predicate for $V_\kappa$, which is a mere set in $V_{\kappa^\Plus}$.

If there are at least two inaccessible cardinals that are not sententially categorical (either first or second order), then $\kappa^\Plus$ will exist and be first-order sententially categorical, where $\kappa$ is the least one, and so there are gaps in the sententially categorical cardinals.
\end{proof}
%
%Notice the subtle issue in the proof, namely, that if we had taken $\theta$ specifically to be the first inaccessible cardinal above $\kappa$ and all $\kappa_\sigma$, then although $V_\theta$ would have known that $\kappa$ is not first-order sententially categorical, it might be that some smaller $\theta'<\theta$, where $V_{\theta'}$ think that $\kappa$ \emph{is} first-order sententially categorical, has some other larger inaccessible cardinal $\kappa'>\kappa$ that it thinks is not first-order sententially categorical.

If there are sufficiently many inaccessible cardinals, then the gaps become arbitrarily complicated and self-reflecting, since if all the gaps in the sententially categorical cardinals had the same simple nature, we could recognize the end of the sententially categorical cardinals---the top gap in a sense---and thereby find an inaccessible cardinal above it that would be characterized by seeing that there was such a new large gap in the sententially categorical cardinals.

A similar analysis works for theory categoricity, even though we lack the analogues of theorems \ref{Theorem.Sentential-categoricity-absolute-to-Vtheta} and \ref{Theorem.Failures-of-sentential-categoricity-smaller} for theory categoricity.

\begin{theorem}\
 \begin{enumerate}
   \item If $\kappa$ is the least inaccessible cardinal that is not theory categorical, either in first order, second order, or $\alpha$-order for some absolutely definable ordinal $\alpha$, then $\kappa$ is fresh in that order.
   \item Consequently, $\kappa^\Plus$, if it exists for this $\kappa$, is first-order theory categorical.
   \item In particular, if there are at least two two inaccessible cardinals that are not theory categorical of a given order, then there are gaps in the theory categorical cardinals.
 \end{enumerate}
\end{theorem}

\begin{proof}
If $\kappa$ is the least cardinal that is not theory categorical, in given absolutely definable order, then $\kappa$ is fresh in that same order, since the theory of $V_\kappa$ could not have arisen earlier, as the smaller cardinals are all theory categorical. So statement 1 holds. Statement 2 now follows as an immediate consequence of theorem \ref{Theorem.If-kappa-fresh-then-kappa-Plus-theory-categorical}. And Statement 3 follows as a consequence of this, since if there are two inaccessible cardinals that are not theory categorical in a given absolutely definable order, then $\kappa^\Plus$ will exist for the smaller of them, and hence be theory categorical in that order by statement 2.
\end{proof}

Note that if there are at least $\continuum^+$ many inaccessible cardinals, then there must be numerous inaccessible cardinals that are not categorical in any given order $\alpha<\continuum$, because there are only continuum many theories (a similar statement holds for larger $\alpha$, if there are at least $(2^\alpha)^+$ many inaccessible cardinals).

\section{On the number of categorical cardinals}

There are, as we have mentioned, at most countably many sententially categorical cardinals, either first or second order, simply because there are only countably many sentences. And it is clear that if there are infinitely many inaccessible cardinals, then there are infinitely many sententially categorical cardinals, because there is the first one, the second, the third and so on, and these are each sententially categorical, each characterized by the statement that there are exactly $n$ inaccessible cardinals. The first $\omega$ many inaccessible cardinals are all first-order sententially categorical in this way.

Similarly, we have mentioned that because there are at most continuum many different theories, there will be at most continuum many theory categorical cardinals. But how many different theory categorical cardinals must there be? If there are uncountably many inaccessible cardinals, must there be uncountably many theory categorical cardinals? Are each of the first $\omega_1$ many inaccessible cardinals theory categorical? If there are at least continuum many inaccessible cardinals, must there be continuum many theory-categorical cardinals? Surprisingly, the answers to all these questions can be negative.

\begin{theorem}\label{Theorem.Countably-many-theory-categorical-cardinals}
It is relatively consistent with \ZFC\ that the inaccessible cardinals form a proper class but there are only countably many theory categorical cardinals.
\end{theorem}

This theorem is an immediate consequence of the following more specific theorem.

\begin{theorem}
Every model of \ZFC\ has a forcing extension, preserving all inaccessible cardinals and creating no new ones, in which there are only countably many second-order theory categorical cardinals.
\end{theorem}

\begin{proof}
Let $G\of\Coll(\omega,\continuum)$ be $V$-generic for the forcing to collapse the continuum to $\omega$, and consider the forcing extension $V[G]$. This forcing is small relative to any inaccessible cardinal, and so they are all preserved in the extension (and forcing can never create new inaccessible cardinals). Furthermore, because this forcing is homogenenous, it follows that the Boolean value $\boolval{\varphi(\check a)}$ of any statement $\varphi$ using only check-name parameters $\check a$ from the ground model is either $0$ or $1$. And since the forcing is definable in $V$, assertions about the Boolean value are expressible in the language of set theory. Indeed, the theory of $V[G]_\kappa$ is the same as the set of sentences $\varphi$ for which the statement `$\boolval{\varphi}=1$' is in the theory of $V_\kappa$ in the ground model. Therefore, the forcing does not create any new instances of sentential or theory categoricity. Since the number of theory categorical cardinals was at most continuum in the ground model, and this cardinal has been collapsed to $\omega$, it follows that there are only countably many theory categorical cardinals in the forcing extension $V[G]$.
\end{proof}

Conversely, there can also be continuum many theory categorical cardinals.

\begin{theorem}\label{Theorem.Continuum-many-theory-categorical-cardinals}
If there are at least $\beta$ many inaccessible cardinals, where $\beta$ is an ordinal below $\continuum^+$, then there is a forcing extension, preserving the continuum and having exactly the same inaccessible cardinals as the ground model, in which the first $\beta$ many inaccessible cardinals are all first-order theory categorical.
\end{theorem}

\begin{proof}
Assume that there are at least $\beta$ many inaccessible cardinals in $V$, where $\beta$ is an ordinal below $\continuum^+$. Let $\kappa_\alpha$ be the $\alpha$th inaccessible cardinal. Similarly, fix an enumeration $\<A_\alpha\mid\alpha<\beta>$ of distinct subsets $A_\alpha\of\omega$ in order type $\beta$, which is possible since $|\beta|$ has size at most continuum. Since each $\kappa_\alpha$ is regular and much larger than $\continuum$, none of the $\kappa_\alpha$ for $\alpha<\beta$ are limits of inaccessible cardinals. Therefore, in every such $V_{\kappa_\alpha}$, the inaccessible cardinals will be strictly bounded below $\kappa_\alpha$. We can therefore force so as to code $A_\alpha$ into the \GCH\ pattern at the first $\omega$ many regular cardinals above the supremum of the inaccessible cardinals below $\kappa_\alpha$. Using an Easton product, we can perform all this forcing at once, making a forcing extension $V[G]$ in which the coding is done for every $\alpha<\beta$, while neither creating nor destroying any inaccessible cardinals. In $V[G]$, the set $A_\alpha$ is definable in $V_{\kappa_\alpha}$ for $\alpha<\beta$, since this model can define the supremum of the inaccessible cardinals below $\kappa_\alpha$ and can observe the \GCH\ pattern at the next $\omega$ many successor cardinals. In particular, the particular sentences saying that the $n$th successor above that supremum does have the \GCH\ or does not are part of the theory of $V_{\kappa_\alpha}$. Since the particular patterns are all different, this means that these cardinals all have different theories. And since they also each think that $A_\alpha$ is not coded in this way into any smaller inaccessible cardinal, these theories will also differ from those of any $V_\kappa$ above every $\kappa_\alpha$. And so all these cardinals are first-order theory categorical in $V[G]$.
\end{proof}

In particular, if there are at least continuum many inaccessible cardinals, then the first continuum many of them can all be theory categorical. If the \GCH\ holds in the ground model, then the forcing also preserves all cardinals and cofinalities.

If we are willing to collapse cardinals, then the same idea shows much more, namely, that we can force any desired ordinal number of inaccessible cardinals to become sententially categorical.

\begin{theorem}
If $\beta$ is an ordinal, of any size, and there are $\beta$ many inaccessible cardinals above $\beta$, then there is a forcing extension $V[G]$, preserving all inaccessible cardinals above $\beta$ and creating no new ones, in which the first $\beta$ many inaccessible cardinals above $\beta$ each become first-order sententially categorical.
\end{theorem}

\begin{proof}
We first force, if necessary, to make $\beta$ a countable ordinal, and then perform Easton forcing so as to arrange the \GCH\ pattern at the cardinals $\aleph_n$ to code a relation on $\omega$ with order type $\beta$. (An essentially similar argument is made in \cite[section~6]{Garland1967:Second-order-cardinal-characterizability}.) Altogether this forcing has size less than the inaccessible cardinals above $\beta$ and so they are all preserved to the resulting forcing extenion $V[G]$ and no new inaccessible cardinals are created. Meanwhile, the ordinal $\beta$ and indeed any particular $\alpha<\beta$ becomes definable in $V[G]$ by means of the \GCH\ pattern on the $\aleph_n$, and any inaccessible cardinal $\kappa$ can observe this coding in $V[G]_\kappa$. Thus, the $\alpha$th inaccessible cardinal $\kappa_\alpha$ in $V[G]$ is categorically characterized by the first-order sentence ``there are precisely $\alpha$ many inaccessible cardinals,'' which for $\alpha<\beta$ is expressible in $V[G]$ by a sentence in the language of set theory.
\end{proof}

That forcing collapsed all cardinals up to $\beta$, making it a countable ordinal. If we would seek to preserve all cardinals up to $\beta$, then we can achieve a similar result with theory categoricity in place of sentential categoricity.

\begin{theorem}
If $\beta$ is an ordinal, of any size, and there are $\beta$ many inaccessible cardinals above $\beta$, then there is a forcing extension $V[G]$, preserving all cardinals up to $\beta$ and preserving all inaccessible cardinals above $\beta$ and creating no new ones, in which the first $\beta$ many inaccessible cardinals above $\beta$ become first-order theory categorical.
\end{theorem}

\begin{proof}
We may force with c.c.c.~forcing so as to push the continuum beyond $\beta$ and then apply theorem \ref{Theorem.Continuum-many-theory-categorical-cardinals}, which will make the first $\beta$ many inaccessible cardinals above $\beta$ all first-order theory categorical, while preserving all cardinals up to $\beta$.
\end{proof}

Let us show next that we can also arrange that the number of theory categorical cardinals is strictly between $\aleph_0$ and the continuum.

\begin{theorem}
It is relatively consistent that the number of first-order theory categorical cardinals is $\omega_1$, even when the continuum is larger than this, and even when the inaccessible cardinals form a proper class.
\end{theorem}

\begin{proof}
By theorem \ref{Theorem.Continuum-many-theory-categorical-cardinals}, we may begin with a model of set theory $V$ in which there are continuum many inaccessible cardinals that are first-order theory categorical. And we may suppose as well that $V$ has abundant inaccessible cardinals, if we like. By forcing if necessary, we may also assume that the continuum hypothesis holds, since this forcing is small, it neither creates nor destroys any inaccessible cardinals, and if the first-order theory categorical cardinals are categorical in the way described in theorem \ref{Theorem.Continuum-many-theory-categorical-cardinals}, by coding reals into the \GCH\ pattern in the blocks below the first continuum (now $\omega_1$) many inaccessible cardinals, then these cardinals will remain first-order theory categorical after forcing the \CH. So we have a model with exactly $\omega_1$ many first-order theory categorical cardinals. We may now force to $V[G]$ by adding any number of Cohen reals via $\Add(\omega,\theta)$, for some definable cardinal $\theta$, such as $\theta=\aleph_3$. This forcing is small, definable and homogeneous, and so it preserves all the inaccessible cardinals, preserves the coding making the first $\omega_1$ many inaccessible cardinals first-order theory categorical, and it creates no new instances of categoricity. So $V[G]$ has the desired features.
\end{proof}

The argument is extremely flexible, and we could have arranged to have exactly $\aleph_{17}$ many first-order theory categorical cardinals, while the continuum is $\aleph_{\omega^2+5}$, or whatever, in diverse other possible combinations.

%\section{Categoricity and forcing}

Let us complete this section with a few observations about the nonabsoluteness of categoricity between a model of set theory and its forcing extensions. Any inaccessible cardinal can be made into the least inaccessible cardinal of a forcing extension (see \cite{Carmody2017:Killing-them-softly}), and this will be first-order sententially categorical. Can we do it, however, while preserving all inaccessible cardinals?

\begin{question}\label{Question.Every-inaccessible-made-categorical}
Can every inaccessible cardinal become first-order sententially categorical in a forcing extension with the same inaccessible cardinals?
\end{question}

Yes, indeed, this is possible.

\begin{theorem}\label{Theorem.Forcing-categoricity}
If $\kappa$ is an inaccessible cardinal, then there is a forcing extension with exactly the same inaccessible cardinals in which $\kappa$ is first-order sententially categorical. Indeed, any countable collection of inaccessible cardinals can be made first-order sententially categorical in a forcing extension with exactly the same inaccessible cardinals.
\end{theorem}

\begin{proof}
To handle just one inaccessible cardinal, we claim that every inaccessible cardinal can be characterized in a forcing extension as the least inaccessible cardinal that is a limit of failures of the \GCH. To see this, suppose that $\kappa$ an inaccessible cardinal. Let $V[C]$ be the forcing extension arising from the forcing that shoots a club set $C\of\kappa$ avoiding the inaccessible cardinals. Conditions are closed bounded sets in $\kappa$ with no inaccessible cardinals, ordered by end-extension. This forcing has $\delta$-closed dense sets for every $\delta<\kappa$, and therefore it adds no new ${<}\kappa$-sequences over the ground model. It therefore preserves $V_\kappa$ and hence all inaccessible cardinals below $\kappa$; and it preserves the inaccessibility of $\kappa$ itself; and being size $\kappa$, it also preserves all larger inaccessible cardinals. In $V[C]$, therefore, we have ensured that $\kappa$ is inaccessible, but not Mahlo, because we now have a club in $\kappa$ avoiding the inaccessible cardinals. In particular, $\kappa$ is a limit of elements of $C$, but no other inaccessible cardinal is a limit of elements of $C$. Let $V[C][G]$ be the subsequent Easton-support forcing extension that forces the \GCH\ up to $\kappa$, except at the successors of cardinals in $C$, where the continuum is the double successor. This forcing preserves all inaccessible cardinals.

In $V[C][G]$, the cardinal $\kappa$ is an inaccessible limit of cardinals at which the \GCH\ fails, namely, the successors of the elements of $C$. But no smaller inaccessible cardinal has that property, because $C$ is bounded below every smaller inaccessible cardinal. So $\kappa$ is the least inaccessible cardinal that is a limit of failures of the \GCH, and this property provides a first-order sentential categorical characterization of $\kappa$ in the forcing extension $V[C][G]$.

To handle countably many, suppose that the order type is $\beta$. By coding into the \GCH\ pattern at the cardinals $\aleph_n$, we can make a real coding $\beta$ first-order definable, and every $\alpha\leq\beta$ will become first-order definable inside every inaccessible $V_\kappa$. Next, we can do the club-shooting trick at each of the $\beta$ many cardinals separately, and so the $\alpha$th cardinal on our list will become the $\alpha$th inaccessible cardinal that is a limit of failures of the \GCH. So all of the $\beta$ many inaccessible cardinals will become first-order sententially categorical in the forcing extension.
\end{proof}

The theorem is an instance of the large cardinal killing-them-softly phenomenon promulgated by Erin Carmody~\cite{Carmody2017:Killing-them-softly}, who proved in a variety of cases that one can often slightly reduce (as little reduced as possible) the large cardinal strength of a large cardinal in a forcing extension. Theorem \ref{Theorem.Forcing-categoricity} achieves this, if we view noncategoricity as a largeness notion---we have killed the noncategoricity of $\kappa$ while preserving its inaccessibility. But one might aspire to sharper (or we should say \emph{softer}) categoricity killing-them-softly results. For example, can we force any inaccessible cardinal that is not second-order theory categorical to become second-order theory categorical, but neither first-order theory categorical nor second-order sententially categorical? Can we force any inaccessible cardinal that is not first-order theory categorical to become first-order theory categorical, but not second-order sententially categorical? Can we force any inaccessible cardinal that is not second-order sententially categorical to become second-order sententially categorical, but not first-order theory categorical? These would be softer killings of noncategoricity than what we achieved in theorem \ref{Theorem.Forcing-categoricity}, and we shall look forward to further investigations of these puzzles.

\section{Complete implication diagram}

Every first-order sententially categorical cardinal is of course also first-order theory categorical, since the sentence is part of the theory, and similarly with second-order; and first-order categoricity immediately implies second-order categoricity for sentences or theories, since first-order assertions count as (trivial) instances of second-order assertions. What we aim to do now is prove that beyond these immediate implications, there are no other provable implications.

\begin{theorem}\label{Theorem.Complete-implication-diagram}
Assuming the consistency of sufficiently many inaccessible cardinals, the complete provable implication diagram for the categoricity notions is as follows:
$$
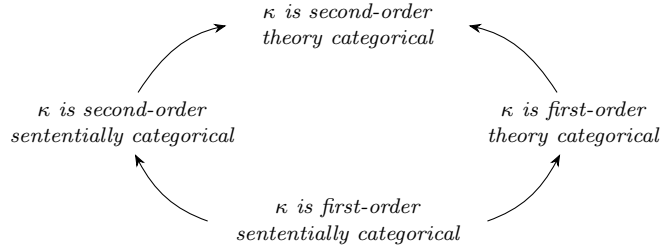
\begin{tikzpicture}[yscale=1.3]\footnotesize
\node (fosc) at (0,0) {\parbox[c]{3.5cm}{\centering $\kappa$ is first-order\\ sententially categorical}};
\node (fotc) at (3,1) {\parbox[c]{3cm}{\centering $\kappa$ is first-order\\ theory categorical}};
\node (sosc) at (-3,1) {\parbox[c]{3.5cm}{\centering $\kappa$ is second-order\\ sententially categorical}};
\node (sotc) at (0,2) {\parbox[c]{3cm}{\centering $\kappa$ is second-order\\ theory categorical}};
\draw[-Stealth] (sosc) to[bend left=20] (sotc.west);
\draw[-Stealth] (fotc) to[bend right=20] (sotc.east);
\draw[-Stealth] (fosc.west) to[bend left=20] (sosc);
\draw[-Stealth] (fosc.east) to[bend right=20] (fotc);
\end{tikzpicture}$$
None of these implications are reversible and no other implications are provable.
\end{theorem}

In fact, we shall prove the following more refined result, which shows that we cannot even get new implications by combining components of the diagram.

\begin{theorem}\label{Theorem.Venn-diagram}
Implications between the various categoricity notions are those shown in the following Venn diagram, and if there are at least $\continuum^+$ many inaccessible cardinals, then every cell of the diagram is inhabited.
$$
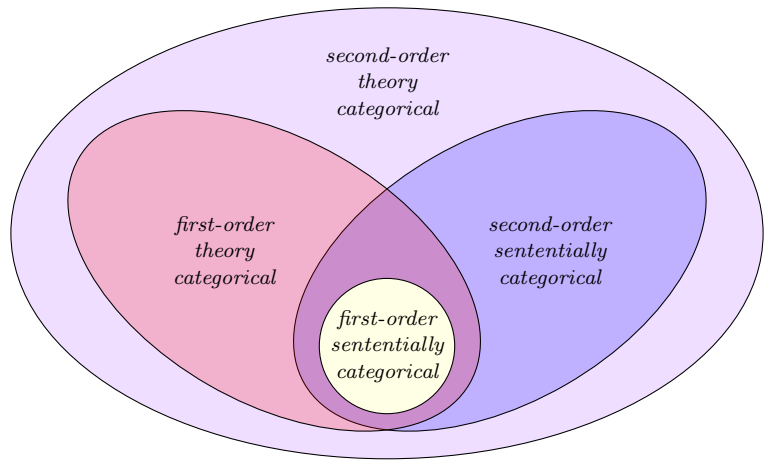
\begin{tikzpicture}[yscale=1]\footnotesize\label{Figure.Venn-diagram}
\draw[fill=Orchid!20] (0,1.5) ellipse [x radius = 5cm, y radius = 3cm];
\draw (0,3.5) node {\parbox[c]{2cm}{\centering second-order\\ theory\\ categorical}};
\draw[fill=blue,fill opacity=.2] (1.5,1) ellipse [x radius = 3cm, y radius = 1.75cm,rotate=30];
\draw[rotate=30] (2.5,0) node {\parbox[c]{2cm}{\centering second-order\\ sententially\\ categorical}};
\draw[fill=red,fill opacity=.2] (-1.5,1) ellipse [x radius = 3cm, y radius = 1.75cm,rotate=-30];
\draw[rotate=-30] (-2.5,0) node {\parbox[c]{2cm}{\centering first-order\\ theory\\ categorical}};
\draw[fill=yellow!10] (0,0) circle [radius = .9cm] node {\parbox[c]{2cm}{\centering first-order\\ sententially\\ categorical}};
\end{tikzpicture}$$
\end{theorem}

The positive implications of the diagram are in each case easy to prove, and these correspond to the inclusions indicated in the Venn diagram. What remains is to prove that all the various cells of the diagram are inhabited. To begin with that, we have noted that if there are any inaccessible cardinals at all, then the least inaccessible cardinal is first-order sententially categorical, and so the yellow region at bottom is inhabited. Next, theorem \ref{Theorem.Least-Mahlo} shows that the least Mahlo cardinal is second-order sententially categorical but not first-order theory categorical, which shows that the blue region at the right is inhabited. But in fact we can weaken the hypothesis necessary for this as follows.

\begin{theorem}\label{Theorem.Second-order-sententially-categorial-not-first-order-theory}
If there is an inaccessible cardinal that is not first-order theory categorical (for example, if there are at least $\continuum^+$ many inaccessible cardinals), then there is an inaccessible cardinal that is second-order sententially categorical, but not first-order theory categorical.
\end{theorem}

\begin{proof}
If there are at least $\continuum^+$ many inaccessible cardinals, then there must be an inaccessible cardinal that is not first-order theory categorical (nor even second-order theory categorical), since there are at most continuum many possible theories. If there is an inaccessible cardinal that is not first-order theory categorical, then there are inaccessible cardinals $\delta<\kappa$, such that $V_\delta$ and $V_\kappa$ have the same first-order theory. Let $\kappa$ be least such that this situation arises. So certainly $\kappa$ is not first-order theory categorical. Nevertheless, the cardinal $\kappa$ is characterized by a certain property of the first-order truth predicate of $V_\kappa$, which is second-order definable. With a single second-order sentence, we can assert that there is some inaccessible cardinal $\delta<\kappa$ for which $V_\delta$ has the same first-order theory as $V_\kappa$, and that $\kappa$ is least for which this situation occurs. So $\kappa$ is second-order sententially categorical.
\end{proof}

Thus, the blue region at the right is inhabited. Next, to show that the red region at the left is inhabited, consider the following theorem.

\begin{theorem}
If there are uncountably many inaccessible cardinals, then there is a first-order theory categorical cardinal that is not second-order sententially categorical.
\end{theorem}

\begin{proof}
This is an instance of theorem \ref{Theorem.Limits-of-sententially-categorical-are-theory-categorical}, but let us give the argument. Suppose that there are uncountably many inaccessible cardinals. Since there are only countably many second-order sentences, there are also only countably many second-order sententially categorical cardinals. In particular, there are only countably many second-order sententially categorical cardinals amongst the first $\omega_1$ many inaccessible cardinals. And so there will be a first inaccessible cardinal $\kappa$ that is larger than all of those. It is part of the first-order theory of $V_\kappa$ that those other smaller sententially categorical cardinals exist, that there are only countably many inaccessible cardinals, and there are no inaccessible cardinals above all of the sententially categorical cardinals. So the theory of $V_\kappa$ characterizes $\kappa$, since no other cardinal can have exactly the same collection of second-order sententially categorical cardinals and view itself as the next inaccessible cardinal after them. So $\kappa$ is first-order theory categorical. Since also it is amongst the first $\omega_1$ many inaccessible cardinals and strictly larger than all sententially categorical cardinals in that interval, it is not itself sententially categorical.
\end{proof}

Let us now prove that if there are sufficiently many large cardinals, then the central dark purple region of the Venn diagram is inhabited. For a first argument, we may modify the Mahlo cardinal argument of theorem \ref{Theorem.Least-Mahlo} by defining that a cardinal $\kappa$ is (first-order) \emph{definably Mahlo}, if every closed unbounded set $C\of\kappa$ that is definable in $V_\kappa$ from parameters in $V_\kappa$ contains a regular cardinal. Another way to say this is that $\kappa$ exhibits the Mahloness property for club sets definable in~$V_\kappa$.

\begin{theorem}\label{Theorem.Least-definably-Mahlo-first-order}
The least inaccessible definably Mahlo cardinal $\kappa$ is second-order sententially categorical and first-order theory categorical but not first-order sententially categorical.
\end{theorem}

\begin{proof}
Let $\kappa$ be the least inaccessible first-order definably Mahlo cardinal. For each natural number $n$, there is by the reflection theorem a definable club of cardinals $\delta<\kappa$ with $V_\delta\elesub_{\Sigma_n}V_\kappa$. Since $\kappa$ is definably Mahlo, this implies that there is some inaccessible cardinal $\delta<\kappa$ with the same $\Sigma_n$ theory as $V_\kappa$. So $\kappa$ is not first-order sententially categorical.

Meanwhile, the fact that $\kappa$ is definably Mahlo is a property of its first-order theory, because every instance of the definably Mahlo scheme is a first-order assertion. And it is also part of the theory of $V_\kappa$ that no smaller $\delta<\kappa$ is definably Mahlo. So $\kappa$ is first-order theory categorical, since no other cardinal can have this combination.

Finally, $\kappa$ is second-order sententially categorical, since the assertion that the theory of $V_\kappa$ contains that combination of statements is a single second-order assertion about $V_\kappa$, namely, the assertion that, ``in the unique truth predicate for first-order truth, every instance of the definably Mahlo scheme comes out true, as well as the assertion that no smaller inaccessible cardinal is definably Mahlo.''

So the least inaccessible definably Mahlo cardinal is second-order sententially categorical, first-order theory categorical, but not first-order sententially categorical.
\end{proof}

Meanwhile, as with theorem \ref{Theorem.Second-order-sententially-categorial-not-first-order-theory}, we can also provide an example from a much weaker large cardinal hypothesis.

\begin{theorem}
If there is an inaccessible cardinal that is not first-order sententially categorical (for example, if there are uncountably many inaccessible cardinals), then the least such cardinal is first-order theory categorical and second-order sententially categorical, but not first-order sententially categorical.
\end{theorem}

\begin{proof}
Suppose that $\kappa$ is the smallest inaccessible cardinal that is not first-order sententially categorical. By theorem \ref{Theorem.Failures-of-sentential-categoricity-smaller}, this means $\kappa$ is smallest with the property that every first-order sentence $\sigma$ true in $V_\kappa$ is also true in some smaller inaccessible $V_\delta$. We claim that $\kappa$ is first-order theory categorical, since it thinks that no smaller inaccessible cardinal has the property we just described, and yet every instance of the defining property of $\kappa$ is part of the first-order theory of $V_\kappa$. So among inaccessible cardinals, only $V_\kappa$ will have that combination in its theory. Finally, we claim also that $\kappa$ is second-order sententially categorical, because first-order truth in $V_\kappa$ is second-order definable, and so the property that every first-order sentence $\sigma$ true in $V_\kappa$ reflects to some inaccessible $V_\delta$ below is a single second-order assertion about $V_\kappa$. So we can characterize $V_\kappa$ by that property plus the assertion that no smaller inaccessible cardinal has that property.
\end{proof}

Thus, if there are sufficiently many inaccessible cardinals, then the central dark purple region of the Venn diagram is inhabited.

Finally, let us show that the light purple region at the top of the Venn diagram also is inhabited.

\begin{theorem}
If there is an inaccessible cardinal that is not second-order theory categorical (for example, if there are at least $\continuum^+$ many inaccessible cardinals), then there is an inaccessible cardinal that is second-order theory categorical, but neither second-order sententially categorical nor first-order theory categorical.
\end{theorem}

\begin{proof}
If there is an inaccessible cardinal that is not second-order theory categorical, then there are inaccessible cardinals $\delta<\lambda$ for which $V_\delta$ and $V_\lambda$ have the same second-order theory. In particular, $\lambda$ has the property (as in theorem \ref{Theorem.Non-theory-categorical-cardinals-are-reflecting}) that for every natural number $n$, there is a smaller inaccessible cardinal $\delta<\lambda$ for which $V_\delta$ has the same $\Sigma^1_n$ theory as $V_\lambda$.

Let $\kappa$ be the smallest inaccessible cardinal with that property, so that for any natural number $n$, there is a smaller inaccessible cardinal $\delta$ for which $V_\delta$ has the same $\Sigma^1_n$ theory as $V_\kappa$. This property is expressible in the second-order theory of the model, and so $\kappa$ is second-order theory categorical. But also, it follows that any particular second-order sentence true in $V_\kappa$ is also true in such a $V_\delta$, and so $\kappa$ is not second-order sententially categorical. But also, we claim, $\kappa$ is not first-order theory categorical, because the first-order theory of $V_\kappa$ is part of the $\Sigma^1_1$ theory of $V_\kappa$, as the truth predicate is definable at this level (indeed, first-order truth has complexity $\Delta^1_1$). Namely, a first-order sentence $\psi$ is true in $V_\kappa$ if and only if there is a class $T$ obeying the Tarskian truth recursion---so it is a truth predicate---according to which $\psi$ is declared true. If $V_\delta$ has the same $\Sigma^1_1$ theory as $V_\kappa$, then they agree on the entire first-order theory, and so $\kappa$ is not first-order theory categorical. So this cardinal $\kappa$ is as desired.
\end{proof}

Thus, we have established theorem \ref{Theorem.Venn-diagram} and therefore also theorem \ref{Theorem.Complete-implication-diagram}.

%\section{Parametric categoricity}
%
%In the main definition, we defined notions of categoricity involving parameter-free sentences and theories. But there are natural notions of categoricity that allow for parameters. For example, let us say an inaccessible cardinal $\kappa$ is first-order \emph{sententially categorical from parameters}, if there is a first-order formula $\varphi$ and parameter $a\in V_\kappa$ such that $V_\kappa\satisfies\varphi[a]$ and for no other inaccessible cardinal $\delta$ does it happen that $a\in V_\delta\satisfies\varphi[a]$. And similarly with second-order and theory categoricity.
%
%\begin{theorem}
%For any inaccessible cardinal $\kappa$, the following are equivalent:
%\begin{enumerate}
%  \item $\kappa$ is not sententially
%  \item mmm
%\end{enumerate}
%\end{theorem}
%

\section{Generalization to other cardinal notions}

Let us prove a version of Zermelo's quasi-categoricity theorem for the class of worldly cardinals, where a cardinal $\kappa$ is \emph{worldly} if $V_\kappa\satisfies\ZFC$, meaning here just the first-order theory \ZFC.

\begin{theorem}
The models of $\ZFC+\textup{Zermelo}_2$, that is, first-order \ZFC\ with second-order Zermelo set theory, are (up to isomorphism) precisely the models $V_\kappa$ for a worldly cardinal $\kappa$.
\end{theorem}

\begin{proof}
If $\kappa$ is worldly, then $V_\kappa$ is a model of \ZFC\ which is correct about power sets, and so it satisfies the second-order separation axiom and hence the second-order Zermelo theory. Conversely, if $M$ is a model of \ZFC\ plus second-order Zermelo set theory, then because of the second-order separation axiom, it will be correct about power sets, and so the model will be well-founded---we may assume without loss that it is transitive---and so the internal computation of the cumulative $V_\alpha$ hierarchy will be correct. So $M=V_\kappa$ for some ordinal $\kappa$ for which $V_\kappa\satisfies\ZFC$, meaning that $\kappa$ is worldly.
\end{proof}

The theory $\ZFC+\textup{Zermelo}_2$ can be equivalently described as $\ZFC+\textup{Separation}_2$, that is, with the second-order separation axiom, since the only part of the second-order Zermelo theory that adds something over first-order \ZFC\ is the second-order separation axiom. All that we really needed, of course, was that the model computes power sets correctly.

Most of the arguments and analysis of this article will simply carry over to the worldly cardinals. More generally, even without an explicit quasi-categoricity result, we may consider the notions of categoricity relative to any fixed class of cardinals $A$. For example, we can define that a cardinal $\kappa$ is \emph{sententially categorical relative to $A$}, if there is a sentence such that $V_\kappa\satisfies\sigma$ and $\kappa$ is only element of $A$ with that feature; and similarly with theory categoricity and so on. In this way, versions of our analysis will apply to the class of weakly compact cardinals, say, or the measurable cardinals or the supercompact cardinals or what have you.

\section{Several philosophical issues}\label{Section.Philosophical-issues}

For the rest of the article, we should like to engage with several matters of a more philosophical nature.

\subsection{Object theory or meta-theory?}

Do we properly consider the categoricity of our mathematical structures as a matter for the set-theoretic object theory, or is categoricity instead an inherently meta-theoretic matter? The question is whether we should undertake an analysis of categoricity as part of the ordinary development of mathematics, for example, as it is undertaken in \ZFC\ set theory, or whether issues of categoricity lie somehow outside of or perhaps prior to the ordinary development of mathematics. Perhaps we best consider the categoricity of our structures such as the natural numbers and the real numbers in a metatheoretic discourse preceeding our ordinary mathematical analysis, in a way that enables our reference to and use of those structures.

Meanwhile, it is certainly possible to mount an analysis of categoricity in the object theory, and we have taken ourselves largely to have done so in the earlier sections of this article. In our main definition, after all, we had stated what it means for a cardinal $\kappa$ to be sententially or theory categorical, providing the definition in terms of the set-theoretic properties of the structure $V_\kappa$ and its competitors $V_\lambda$, referring to the features of these structures that are revealed in their first and second-order theories. These notions are all first-order expressible in the background set theory, and so the entire discussion can be seen as taking place within \ZFC\ set theory. In that theory we can express what it means for a cardinal $\kappa$ to be inaccessible, what it means to have the corresponding rank-initial segment $V_\kappa$, what the first and second-order theories of this structure are, and so on, interpreting second-order quantifiers as quantifying over the subsets of $V_\kappa$, which amounts to a first-order quantification in the background set theory. In this way, the property of a cardinal $\kappa$ being first-order sententially categorical or second-order theory categorical and so on are seen as part of the ordinary set-theoretic development, just like any of the other large cardinal properties commonly considered in set theory.

This choice has consequences. Precisely because we undertook our analysis in the object theory, the notions of categoricity that we provided can be interpreted inside any model of \ZFC, since every such model provides its notions of what it means to be an inaccessible cardinal and what the truth predicates are like for the structures $V_\kappa$ and so on. Thus, categoricity becomes a relative concept, relative to a given set-theoretic background. The categoricity of a cardinal $\kappa$ inside a model $M\satisfies\ZFC$ depends in large part on what other kinds of inaccessible cardinals are available inside $M$. The categoricity of a structure in a model of set theory ultimately depends on what other structures are available in that structure. In light of theorem \ref{Theorem.Every-theory-can-be-categorical} and the subsequent remarks, a structure may admit a categorical characterization inside a particular model of set theory $M$ but not outside.

The same object-theory/meta-theory issue arises also with how we are to interpret Zermelo's quasi-categoricity theorem. Namely, are we to take this result as an ordinary mathematical development within the object theory of \ZFC\ or alternatively as a fundamentally meta-theoretic observation? Most contemporary set theorists appear to understand theorem \ref{Theorem.Zermelo-quasi-categoricity} entirely in the object theory, as a theorem of \ZFC, an early part of the development of large cardinals within \ZFC\ set theory. On this reading of the theorem, the second-order quantifiers of the theory $\ZFC_2$ in a model $V_\kappa$ are taken simply to range over the subsets of $V_\kappa$ in the ambient set-theoretic context.

Mathematicians often also adopt the corresponding stance toward Dedekind's categorical characterization of the natural number structure $\<\N,S,0>$ and Huntington's categorical characterization of the real number field as the unique complete ordered field. In each case, the second-order theory is interpreted within the first-order set-theoretic object theory. In this way, the various categoricity and quasi-categoricity results become theorems of \ZFC\ set theory, able to be applied inside any particular model of \ZFC.

An opposing philosophical perspective on these results taken by many, such as \cite{Kreisel1967:Informal-rigour-and-completeness-proofs, Isaacson2011:TheRealityOfMathematicsAndTheCaseOfSetTheory}, would be to insist that the second-order theories are to be interpreted in a true second-order logic, where the second-order quantifiers range through all of the actual subsets of the domain, and not merely those available within the limited context of a particular model of \ZFC. On this view, some models of \ZFC\ are simply wrong about which subsets there are and so they do not reliably interpret the second-order theories. On this perspective, the categoricity results are seen as more fundamental observations about those categorical structures, taking place outside the set-theoretic object theory. Zermelo's quasi-categoricity result, for example, is taken to be telling us about the nature of the actual set-theoretic universe, rather than just something that happens inside various \ZFC\ models. We are interested in using the second-order extensions of $\ZFC_2$ to tell us about the nature of the full set-theoretic universe $V$, not merely some set-sized models $V_\kappa$.

This difference in perspective is therefore tightly connected with the question of whether second-order logic stands on its own or is to be regarded as interpreted inside set theory. See Button and Walsh \cite{Button+Walsh2018:Philosophy-and-model-theory} for further discussion of the spectrum of philosophical attitudes towards categoricity.

\subsection{Unifying the approaches}

Let us argue, however, that the opposing perspectives are not ultimately so different after all. Suppose that we have taken a metatheoretic approach to categoricity, interpreting our second-order theories in what we think is the full, robust, true second-order logic.

As we reason in this second-order logic, we shall naturally find ourselves eventually wanting to appeal to various principles of second-order logic that we expect to be true of it---we shall likely commit to diverse natural second-order set-existence principles. Of course, in light of the necessary failure of compactness for second-order logic, however, we shall never achieve a truly sound and complete axiomatization and proof system for this second-order logic; but we shall articulate more and more the nature of the set-existence principles of our intended second-order logic.

A neutral observer of this process could remark that what we would seem to be doing is essentially putting forth a theory of sets in the metatheory to govern the interpretation of second-order logic in the object theory. Specifically, we would have what amounts to a first-order theory of sets adopted in the metatheory to provide the sets that will be interpreted as classes in the second-order object theory.

In this way, the two approaches we mentioned earlier increasingly resemble each other (a similar conclusion is made in \cite{Väänänen2012:Second-order-logic-or-set-theory}). The set-theorist who develops notions of categoricity and quasi-categoricity in the first-order object theory of \ZFC\ can be seen simply as having adopted that theory to govern the set-existence principles underlying the interpretation of second-order logic that this theory provides. And the logician who claims to be using a pure second-order logic finds him or herself increasingly articulating set-existence principles to govern that logic, and can therefore be seen essentially to be presenting a first-order set-theoretic account of the intended set-theoretic background. In the end, they are both in the same place---what is the set-theoretic background object theory for one set theorist is simply the second-order metatheory of another.

\subsection{A plurality of metatheoretic contexts}

The first author has described how this kind of move, transforming object theory to metatheory and vice versa, is a natural outcome of set-theoretic pluralism:
\begin{quote}\small
The multiverse [i.e. pluralist] perspective ultimately provides what I view as an enlargement of the theory/metatheory distinction. There are not merely two sides of this distinction, the object theory and the metatheory; rather, there is a vast hierarchy of metatheories. Every set-theoretic context, after all, provides in effect a metatheoretic background for the models and theories that exist in that context---a model theory for the models and theories one finds there. Every model of set theory provides an interpretation of second-order logic, for example, using the sets and predicates existing there. Yet a given model of set theory $M$ may itself be a model inside a larger model of set theory $N$, and so what previously had been the absolute set-theoretic background, for the people living inside $M$, becomes just one of the possible models of set theory, from the perspective of the larger model $N$. Each metatheoretic context becomes just another model at the higher level. In this way, we have theory, metatheory, metametatheory, and so on, a vast hierarchy of possible set-theoretic backgrounds. % This point of view amounts to a denial of the need for uniqueness in Maddy's metamathematical corral conception; it seems to enlarge our understanding of metamathematical issues when different metamathematical contexts offer competing metamathematical conclusions and when the independence phenomenon reaches into metamathematical contexts.
\cite[p.~298]{Hamkins2021:Lectures-on-the-philosophy-of-mathematics}
\end{quote}
The plurality of set-theoretic contexts for a given structure thus reveals ultimately a measure of nonabsoluteness of categoricity---whether a structure is categorical or not depends on the set-theoretic background in which the second-order logic is interpreted. Theorem \ref{Theorem.Every-theory-can-be-categorical} shows, after all, that any sentence or theory that can hold in a model of $\ZFC_2$ can also serve as a categorical characterization in the context of a suitably chosen set-theoretic background. In this way, discussions of categoricity become wrapped up with the debate on set-theoretic pluralism.

Much of our earlier analysis, such as theorems \ref{Theorem.Kappa^Plus-sententially-categorical}, \ref{Theorem.If-kappa-fresh-then-kappa-Plus-theory-categorical}, and \ref{Theorem.Limits-of-sententially-categorical-are-theory-categorical}, reveal how categoricity tends to push us out of any given level of the set-theoretic universe. If one model $V_\kappa$ has a categorical characterization, then so does the next one $V_{\kappa^\Plus}$, and it interprets the previous one, but not conversely. If foundations is about maximizing interpretative power, then we are thereby pushed out of any given categorical model. No one categorical structure will ever be enough, because if we have a structure and its categorical characterization, then we also have a truth predicate for that structure, and this is also categorical, but not interpretable in the original structure. Väänänen \cite[proposition~3,4]{Väänänen2012:Second-order-logic-or-set-theory} observes that if $\kappa<\lambda$ are sententially categorical, then the second-order theory of $\<V_\kappa,\in>$ is Turing computable from that of $\<V_\lambda,\in>$, a simple consequence of the fact that $V_\kappa$ is definable in $V_\lambda$, and so assertions about $V_\kappa$ amount to relativized truth assertions in $V_\lambda$. Indeed, the argument does not require $\lambda$ to be characterizable, and furthermore it shows that the second-order theory of $V_\kappa$ is Turing computable from the first-order theory of $V_\lambda$. Because of this reduction in complexity, it follows that the second-order theory of $V_\lambda$ is not Turing computable from that of $V_\kappa$. Thus, as we move to higher and higher categorical cardinals, the theory necessarily becomes more complex. We shall be eternally pushed toward larger and larger categorical contexts, truth, then truth-about-truth, truth-about-truth-about-truth, the next inaccessible beyond, and so on. We shall never be done.

\subsection{Categoricity as semantic completeness}

A categorical theory completely determines the structure in which it holds, and in this sense, the theory also completely determines the truths of that structure. If $T$ is categorical, after all, then for any assertion $\varphi$ in that language, either $T\satisfies\varphi$ or $T\satisfies\neg\varphi$, precisely because either $\varphi$ is true in the unique model of $T$ or it isn't. Dedekind arithmetic, for example, is complete in this sense for arithmetic assertions, and the axioms of a complete ordered field are complete for assertions about the real numbers.

Kreisel \cite{Kreisel1967:Informal-rigour-and-completeness-proofs} argued similarly with second-order set theory, using Zermelo's quasi-categoricity result to argue that $\ZFC_2$ settles the continuum hypothesis. Since the truth or falsity of the continuum hypothesis is revealed very low in the set-theoretic hierarchy, at the level of $V_{\omega+2}$, and since all the models of $\ZFC_2$ have the form $V_\kappa$ for an inaccessible cardinal $\kappa$ and these agree on $V_{\omega+2}$, it follows that all the models of $\ZFC_2$ give the same answer for the continuum hypothesis. In this sense, $\ZFC_2$ settles the continuum hypothesis. Daniel Isaacson \cite{Isaacson2011:TheRealityOfMathematicsAndTheCaseOfSetTheory} also defends this view.

Similar reasoning shows that $\ZFC_2$ is complete with respect to nearly the entirety of classical mathematics, which takes place at comparatively low levels of the set-theoretic hierarchy---mathematicians have argued that $V_{\omega+5}$ is sufficient, but indeed even $V_{\omega+\omega}$ or $V_{\omega_1}$ would be good enough---the argument shows that $\ZFC_2$ is semantically complete with respect to any mathematical question that can be resolved inside any $V_\alpha$ up to the first inaccessible cardinal.

So it would seem to be great news---our fundamental theory $\ZFC_2$ determines the answer to essentially every mathematical question! Fantastic! Let's get straight to work with this theory.

But wait, you say that it isn't working? We are told that the theory $\ZFC_2$ determines the answer to CH and all other classical mathematical questions, but the disappointment comes when we seem unable to use this theory in any way to figure out what the answers actually are. The reason is that we lack a sound, complete, and verifiable proof system for second-order logic, and so we cannot actually use the completeness of the theory $\ZFC_2$ in any mechanistic manner of reasoning to determine the answer. When working with a second-order theory what often happens in practice is that one adopts as much of the second-order theory as one can, by gathering together the set-existence principles one views as sound. But to do so is as we explained earlier to adopt in the metatheory what amounts to a first-order set theory such as \ZFC. And since this theory does not settle the continuum hypothesis or even every arithmetic question---it must be incomplete---we are forced to give up the semantic completeness claim.

The completeness of a first-order theory $T$ (specified by a computable list of axioms) leads necessarily to a computable decision procedure for the entire content of the theory. Namely, given any question $\varphi$, we can search systematically for a proof $T\proves\varphi$ or a refutation $T\proves\neg\varphi$; if the theory is complete, then we will eventually find one of these and thereby come to the answer of whether $\varphi$ holds in the theory or not. In second-order logic, however, we have no such complete proof system and we must remain basically at a loss. For this reason, the semantic completeness of our second-order set theories is not as useful as it might seem.

The objection we have made so far to the completeness claim for the second-order theory is about our inability to use it, rather than an ontological point about what there is and what is true. So let us now mount a sharper objection. Namely, we claim that one cannot deduce the definiteness of our mathematical structures on the basis of categorical characterizations in second-order logic. What we claim is that any legitimate metatheoretic aparatus, whether it is second-order logic, plural quantifiers, Fregean concepts or what have you, faces a certain dichotomy---either it will be incomplete in the metatheoretic account it provides, like using a first-order commitment such as \ZFC, or else it will be complete, but in a way that begs the question concerning the definite nature of the metatheoretic ontology. This is question-begging because we cannot establish the definiteness of the object-theory set concept by appealing to a presumed definiteness of the meta-theory set concept. To do so is merely to put off the definiteness objection from the object theory to the metatheory, but without in any way answering that objection. If someone says that our concept of set is definite and complete because they have a categorical account of it in second-order logic, then of course we would simply want to know why their metatheoretic concept of set (or of pluralities or Fregean concepts or what have you) is definite and complete.

The first author argued similarly as follows that the semantic completeness of the second-order theory $\ZFC_2$ is illusory, for all that has happened is that we have pushed off the incompleteness into the metatheory.
\begin{quote}\small
Critics view this [Kreisel's argument on the determinateness of CH] as sleight of hand, since second-order logic amounts to set theory itself, in the metatheory. That is, if the interpretation of second-order logic is seen as inherently set-theoretic, with no higher claim to absolute meaning or interpretation than set theory, then to fix an interpretation of second-order logic is precisely to fix a particular set-theoretic background in which to interpret second-order claims. And to say that the continuum hypothesis is determined by [the] second-order set theory is to say that, no matter which set-theoretic background we have, it either asserts the continuum hypothesis or it asserts the negation. I find this to be like saying that the exact time and location of one's death is fully determinate because whatever the future brings, the actual death will occur at a particular time and location. But is this a satisfactory proof that the future is ``determinate''? No, for we might regard the future as indeterminate, even while granting that ultimately something particular must happen. Similarly, the proper set-theoretic commitments of our metatheory are open for discussion, even if any complete specification will ultimately include either the continuum hypothesis or its negation. Since different set-theoretic choices for our interpretation of second-order logic will cause different outcomes for the continuum hypothesis, the principle remains in this sense indeterminate. \cite[p.~288]{Hamkins2021:Lectures-on-the-philosophy-of-mathematics}
\end{quote}\goodbreak

We should like to emphasize that this objection applies just as much to the more ordinary categorical characterizations of mathematical structure. The fact that the Dedekind axioms for the natural numbers determine a definite structure $\<\N,S,0>$ and therefore determine all the arithmetic truths is not actually helpful for us to discover those truths. Ultimately, number theorists will find themselves adopting what amounts to a first-order theory such as PA or ZFC in the metatheory, and these remain incomplete for arithmetic truth. For analogous reasons, the categorical accounts of the structures $V_\kappa$ on which we have focussed in this article may ultimately be less fulfilling than one hoped.

\subsection{Categoricity, reflection and realism}

Finally, we should like to call attention to a certain perplexing tension we observe between two fundamental values in mathematics---the contradictory natures of categoricity and set-theoretic reflection. The matter deserves philosophical attention.

On the one hand, mathematicians almost universally seek categorical accounts of their fundamental mathematical structures, from Dedekind's axiomatization of arithmetic to the characterization of the real numbers as a complete ordered field. Categoricity is taken as a positive value and a key general goal in mathematical practice. At least part of the explanation for this (our criticism notwithstanding) is that the categorical characterizations of our structures seem to give us reason to regard these structures as definite. We know what we mean by the natural numbers, on this view, precisely because we can categorically describe the natural number structure. Indeed, because all our fundamental mathematical structures admit of such categorical characterizations, we thereby have reason to think of them as definite and real, and in this way categoricity seems to lead to mathematical realism. At the same time, categoricity seems also to implement structuralism, because the categorical accounts of our fundamental structures invariably do so only up to isomorphism, and so to regard every structure that fulfills the characterization as perfectly satisfactory is precisely to adopt the structuralist stance.

%Perhaps this shows that the categorical large cardinals, either by sentences or theories, exhibit fundamental mathematical structure, giving similar grounds for realism about those cardinals.

On the other hand, set theorists vigorously defend principles of set-theoretic reflection, asserting in various ways that every truth of the full set-theoretic universe reflects down to the same truth made in a set-sized structure. Reflection is often described as expressing a core feature of the set-theoretic universe, and indeed the \Levy-Montague reflection theorem is equivalent over a weak theory to the replacement axiom of \ZFC. Reflection ideas are used not only to justify the \ZFC\ axioms of set theory, but also the existence of large cardinals \cite{Reinhardt1974:RemarksOnReflectionPrinciplesLargeCardinalsAndElementaryEmbeddings, Maddy1988:BelievingTheAxiomsI}.

The puzzling conflict we aim to highlight between categoricity and reflection is that reflection is at heart an anti-categoricity principle---it asserts explicitly that no statement characterizes the set-theoretic universe $V$, because every statement true in $V$ is also true in a much smaller structure. The philosophical question to sort out here is how we can regard categoricity as vitally important in all our fundamental mathematical structures and yet simultaneously assert as a core principle that the set-theoretic universe itself is not categorical. Ultimately, there must be a fundamental mis-match between the extent of the reflection phenomenon and the complexity of any categorical characterization of the set-theoretic universe, since the kinds of statements and theories that reflect clearly cannot encompass the categorical characterization itself, which by the fact of categoricity does not reflect to any smaller structure.

The tension manifests also in attitudes toward large cardinals. Set theorists commonly defend a larger-is-better approach to large cardinals, pointing to the highly structured tower of consistency strength that they provide and the explanatory consequences down low of even the strongest large cardinal notions. Yet, as we have mentioned, categoricity for large cardinals is a smallness notion rather than a largeness notion. Theorem \ref{Theorem.Correctness-relations} shows that the categorical cardinals are all below the least $\undertilde\Sigma_2$-correct cardinal, and consequently below every strong cardinal, every supercompact cardinal, every extendible cardinal, every totally otherworldly cardinal and more. If we look upon categoricity as desireable in a foundational theory, therefore, we would seem to be pushed toward the low end of the large cardinal hierarchy, to the smallest large cardinals, to the set-theoretic universes satisfying categorical theories. For example, the theory $\ZFC_2+$``there are no inaccessible cardinals'' is categorical, as is $\ZFC_2+$``there are exactly $\omega^2+5$ inaccessible cardinals.'' But in the philosophy of set theory, one finds instead general arguments against such theories in the foundations of set theory---they are viewed as restrictive and limiting. Penelope Maddy \cite{Maddy1998:V=LAndMaximize}, for example, formulates the \emph{maximize} principle and uses it to explain set theorist's resistance to the axiom of constructibility and to large cardinal nonexistence axioms in general on the grounds that they are restrictive. According to \emph{maximize}, we should rather seek open-ended, nonlimiting axiomatizations of set theory. Even critics of Maddy's position, such as the first author in \cite{Hamkins2014:MultiverseOnVeqL}, retain an open-ended conception of set theory and do not push for categoricity in the set-theoretic universe.

Theorem \ref{Theorem.Correctness-relations} and the idea to which it leads, that categoricity is for small universes only, seem to suggest that we might not want or expect a categorical account of the full set-theoretic universe $V$, nor indeed even freshness for the theory of $V$. Perhaps this is the essence of reflection, that whatever is true in $V$, including perhaps the entire theory of $V$, has already been seen before many times along the way.
% It may be interesting to note in connection with this that Kelley-Morse set theory, because it proves the existence of a first-order truth predicate for $V$, also proves that there is a closed unbounded proper class of cardinals $\kappa$ forming an elementary chain to $V$ with respect to first-order assertions. In this sense, \KM\ proves that $V$ is not fresh, even with any arbitrary set of parameters.

According to the \emph{toy model} perspective (described in \cite{Hamkins2012:TheSet-TheoreticalMultiverse}, \cite{Hamkins2014:MultiverseOnVeqL}), one studies the various set models of set theory and how they relate to one another partly in order to gain insight into the nature of the larger actual set-theoretic universe in the context of its multiverse including all its various forcing extensions. The toy models serve as a proxy for the real thing, which remains inaccessible to us---we look into the toy models to learn what might be true in $V$ or what we would desire to see in $V$. When we consider the toy models of $V_\kappa$ for inaccessible cardinals $\kappa$, we see that it is only the smallish large cardinals that realize a categorical theory, while the larger large cardinals do not, and so the toy model perspective together with the \emph{maximize} maxim seems to incline us to think that the full universe $V$ should not fulfill a categorical or even a fresh theory. On the toy model perspective, we might adopt noncategoricity and nonfreshness as a goal.

Similarly, one might be led by the philosophical reflection arguments to principles such as ``$\Ord$ is Mahlo,'' the scheme asserting that every first-order definable class club of ordinals contains a regular cardinal. It then follows directly from this principle that the universe is not first-order sententially categorical, since any statement true in $V$ would reflect to many inaccessible levels $V_\kappa$. Going beyond this, consider a set theorist with the universist perspective, who holds that there is a unique set-theoretic universe containing all sets, that they have a certain definite existence there, and that there is a definite nature for set-theoretic truth. In this case, we would expect a class truth predicate for first-order truth (and the existence of such a class is provable in Kelley-Morse set theory). If the reflection ideas underlying ``$\Ord$ is Mahlo'' encompassed definitions allowing this class as a parameter, then we would immediately get unboundedly many inaccessible reflecting cardinals $V_\kappa\elesub V$. Consequently, the universe would not satisfy a categorical theory. In this way, again, the philosophical reflection and definiteness arguments push one towards noncategoricity in set-theoretic truth.

Väänänen \cite[\S4.1]{Väänänen2012:Second-order-logic-or-set-theory} makes the point that no one categorical structure can interpret all the others, and in this sense no one categorical structure can serve as a foundation of mathematics, again pushing us to non-categorical foundations.

Because of these tensions between noncategoricity and the universe view in set theory, we claim that noncategoricity tends to undermine the idea of the set-theoretic universe as a unique fully completed set-theoretic realm. If the ultimate set theory is not categorical, after all, then whatever set-theoretic truths we might assert of the final set-theoretic universe will also be true in other distinct set-theoretic realms. Thus, the noncategoricity of the theory by its very nature leads one to a form of ontological pluralism for set theory. The same set-theoretic truths will be true elsewhere. Whatever it is that the universist believes to individuate the set-theoretic universe $V$ cannot be expressible as part of the theory of this set-theoretic realm.

As we see it, the philosophical problem here is to explain this transition in attitude toward categoricity. Why do we take categoricity as a fundamental value for smallish mathematical structures such as the natural numbers, the real numbers and so on, but not for the set-theoretic universe as a whole?

Let us try to offer a solution by describing a way out of the impasse, even though ultimately we shall take a different lesson from this analysis. What we claim is that if one takes second-order logic to have a fixed meaning, one where the metatheoretic concept of set obeys something at least like \ZFC, then the set-theoretic universe $V$ does in fact have a categorical characterization in second-order logic. Furthermore, the existence of this characterization ultimately places limits on the extent of reflection that is possible for the set-theoretic universe to exhibit.

Specifically, we claim, the class structure of the set-theoretic universe $\<V,\in>$ is characterized up to isomorphism in second-order logic as the unique well-founded extensional set-like relation realizing every subset of its domain. In more detail:
\begin{enumerate}
 \item The membership relation $\in$ is extensional.
         $$x=y\ \iff\ \forall z\, (z\in x\iff z\in y)$$
 \item The membership relation $\in$ is well founded.
        $$\forall A\,\bigl[\exists x\, Ax\implies\exists x\, \bigl(Ax\wedge\forall y\,(y\in x\implies \neg Ay)\bigr)\bigr]$$
 \item The membership relation $\in$ is set-like.
       $$\forall a\exists A\, \forall x\, \bigl(x\in a\iff Ax\bigr)$$
 \item Every subset of $V$ is realized.
         $$\forall A \exists a \forall x\, \bigl(x\in a\iff Ax\bigr)$$
\end{enumerate}
The lower-case quantifers $\forall x$ quantify over the objects $x$ of the domain $V$, while the upper-case quantifiers $\forall A$ are interpreted in second-order logic, ranging over all subsets of $V$. Note that $V$ itself will be a proper class in the metatheory, not a set, and so we emphasize that $\forall A$ means for all sub\emph{sets} $A\of V$, which will not include the proper class subclasses of $V$, such as $V$ itself or the class of ordinals of $V$. So this quantifier differs from the class quantifier used in \Godel-Bernays and Kelley-Morse set theory, since that quantifier ranges over all classes including proper classes, not just sets. Axiom (2) is the assertion that every nonempty subset $A\of V$ has an $\in$-minimal element. Axiom (3) asserts that the $\in$-members of any set in $V$ form a set---that is, a set in the metatheory, using the second-order concept of set. Axiom (4) asserts conversely that for every subset $A\of V$ there is an object $a$ in $V$ whose $\in$-elements are the same the elements of $A$. In this way, axioms (3) and (4) assert a kind of correspondence between the object theory and the metatheory as to what the sets are. Note also that axiom (4) makes a stronger claim than the second-order separation axiom $\textup{Separation}_2$, which asserts merely that every subset of a set already in $V$ is in $V$, whereas our axiom asserts that every subset of $V$ is in $V$. So this theory implies $\textup{Replacement}_2$ and consequently $\ZFC_2$ and more.

The characterization is entirely about the interplay of the metatheoretic and object-theoretic concepts of set, and it therefore relies critically on our having already fixed an interpretation of second-order logic. At bottom what the axiomatization asserts---perhaps disappointingly---is that the set-theoretic structure $\<V,\in>$ implements up to isomorphism exactly the wellfounded cumulative set hierarchy from the metatheory into the object theory. At bottom, it copies the metatheory to the object theory.

To begin the proof of categoricity, let us observe how the process of closing under all subsets---specified by axiom (4)---proceeds via the cumulative hierarchy. The axioms successively force certain kinds of objects into $V$ in a transfinite recursive process. Namely, we begin with nothing $V_0=\emptyset$, and at successor stages $V_{\alpha+1}$ has objects realizing every possible subset of $V_\alpha$; so it is in effect (isomorphic to) the full, actual power set of $V_\alpha$ as provided by the metatheoretic set concept interpreting second-order logic. At limit stages $\lambda$, we gather together everything we've added so far $V_\lambda=\Union_{\alpha<\lambda}V_\alpha$. The point is that if this is a set, then the closure process of axiom (4) asserts that $V_\lambda$ itself must be realized by an element of $V$, and so the closure process will continue to $V_{\lambda+1}$ and so on. Ultimately, we build the $V_\alpha$ hierarchy through all the ordinals of the metatheory and the resulting universe is $V=\Union_\alpha V_\alpha$. This structure fulfills axioms (1) and (2) by construction, since we assume the metatheoretic set concept is wellfounded and extensional; it fulfills axiom (3) because we only ever added objects that were sets in the metatheory; and it fulfills axiom (4) since every subset will be bounded in rank, since we presume that the metatheoretic set concept fulfills the replacement axiom.

To complete the proof of categoricity is now a simple induction on rank. Namely, if we have two structures $\<V,\in^V>$ and $\<W,\in^W>$ satisfying axioms (1), (2), (3) and (4), then we can construct an isomorphism of the corresponding levels $V_\alpha\iso W_\alpha$ of the cumulative hierarchy. Since the membership relations are both well founded, the two conceptions of ordinals and their stages in the cumulative hierarchy will both conform with the metatheoretic conception of ordinal. Both hierarchies begin with nothing, and  if we have an isomorphism at level $\alpha$, then it extends uniquely to an isomorphism at level $\alpha+1$, since any subset of $V_\alpha$ can be transferred by our partial isomorphism to a subset of $W_\alpha$ and vice versa. And the partial isomorphisms union to an isomorphism at the limit stages. It cannot be that one structure stops growing at a stage while the other continues, since the ordinal heights of the two stages are isomorphic and so one of them is a set in the metatheory if and only if the other also is. Finally, this cumulative recursive process must ultimately exhaust the objects of $V$ and $W$, since if there is an object $a$ in $V$ outside the $V_\alpha$ hierarchy, then its members form a set by axiom (3), and its members-of-members and so on, but this iterated collection (which forms a set in the metatheory since we assume \ZFC\ in the metatheory) cannot have an $\in^V$-minimal element $b$ outside the $V_\alpha$ hierarchy, contrary to axiom (2), since any such object $b$ would have all its $\in^V$-members inside the $V_\alpha$ hierarchy, which would place $b$ into the hierarchy at the supremum of those stages. Thus, the cumulative $V_\alpha$ and $W_\alpha$ hierarchies exhaust all of $V$ and $W$ and we therefore have an isomorphism of $\<V,\in^V>$ with $\<W,\in^W>$. So this is a categorical characterization. This categoricity argument is essentially similar to the central argument of \cite{Martin2001:MultipleUniversesOfSetsAndIndeterminateTruthValues}, which can be read as a categoricity thesis.

The key difference between our categorical account and the second-order set theory $\ZFC_2$ is that $\ZFC_2$ has only the second-order separation axiom rather than our axiom (4). The second-order separation axiom is insufficient to force the cumulative hierarchy to keep growing past inaccessible cardinals, since $V_\kappa$ for $\kappa$ inaccessible satisfies the second-order separation axiom. But such a structure $V_\kappa$ does not satisfy our axiom (4), since $\kappa$ and $V_\kappa$ are sets in the metatheory, and so must be added as elements, causing the hierarchy to keep growing. This key difference in the axioms is why Zermelo achieves only a quasi-categoricity result, making stops at every inaccessible cardinal, whereas we are able to achieve full categoricity, proceeding onward and upward to the full set-theoretic universe.

Note also the difference in kind between the categorical account we have provided above for the set-theoretic universe $V$ and the notion of categoricity used for categorical cardinals in the main definition. When defining the categorical cardinals, although we had used a second-order theory for the structure $V_\kappa$ in question, ultimately this amounts to a first-order notion in the set-theoretic universe $V$ in which the definition is considered, since we are in effect quantifying over $V_{\kappa+1}$, which is a set in $V$. The categorical account of $V$, in contrast, is not first-order in $V$, but second-order over $V$. For this reason, the notion of categoricity provided in the categorical account of $V$ is not subject to the conclusion of theorem \ref{Theorem.Correctness-relations}.

Does this categoricity proposal resolve the tension between categoricity and reflection? By providing a categorical account of the set-theoretic universe in second-order logic, it seems both to place categoricity as the more primary notion, and also to identify limitations on the possible extent of reflection. The reflection principle cannot rise fully to second-order logic, since one of the truths of the set-theoretic universe is that it is not a set, and this is a truth that cannot reflect to any actual set.

But is this categorical account of the set-theoretic universe satisfactory? Does it enable us to secure a definite meaning for the universe of sets and definite account for which sets there are? No, not really. Let us criticise it. As we have mentioned, the characterization proceeds essentially by copying the metatheoretic concept of set into the object-theoretic account. The categorical account is therefore only sensible when we have already fixed a meaning of second-order logic, when we have in effect fixed a set concept in the metatheory. And if we had used a different metatheoretic concept of set, for example, if a form of set-theoretic pluralism were the case, then the categorical characterization undertaken with that other concept of set would give rise to \emph{that} concept of set as the unique set-theoretic realm. For this reason, the categorical characterization by itself is without power to refute pluralism or to tell us which sets there really are. We cannot establish the definiteness of our set concept on the basis of the categoricity result without presuming the definiteness of the set concept arising from the interpretation of second-order logic in the metatheory. To attempt to do so would be circular reasoning that merely pushes off the problem from the object theory to the metatheory.
%But what good is a characterization of the set-theoretic universe if at bottom it asserts essentially only that ``whatever sets there are, those are the ones that exist in my set-theoretic universe''?

The problem is fundamentally similar to the objection we had raised earlier to Kreisel's observation that the continuum hypothesis is settled in second-order set theory. We objected that the argument is circular, because Kreisel's observation shows merely that the \CH\ is settled, provided that one has already fixed a complete concept of set to be used in the metatheoretic interpretation of second-order logic. But if there were a choice of such metatheoretic set concepts, then it wouldn't necessarily be settled the same way by all of them. This is the same circularity that seems to prevent us from using the categorical account of the set-theoretic universe to come to an understanding of which sets there are.

%
%once one has fixed the interpretation of second-order logic, then the continuum hypothesis is settled by it. Such a line of reasoning cannot seem to establish the definiteness of \CH, however, unless one knows that there is only one possible concept of set to use in the metatheory. Similarly, the categorical characterization of the set-theoretic universe cannot be used to establish the singularity of the set concept, since if there were a plurality of set concepts to be used as the metatheoretic interpretation of second-order logic, then each of them would lead to itself as the unique definite concept of set.
%
%Namely, the reader can readily observe that all we have really done with the categorical account is to put off the problem of characterizing the set-theoretic universe from the object theory to the metatheory. In order to make sense of the second-order characterization, we had to assume that there was already fixed an interpretation of second-order logic. But of course, this amounts to having fixed a concept of set in the metatheory. For this reason, the categorical characterization cannot ultimately help us to make definite the concept of set---it is ultimately circular.

The main lesson we propose to take from the categorical characterization of the set-theoretic universe $\<V,\in>$ is that the obvious circularity and uselessness of it help to show how the other second-order characterizations that we have in mathematics are similarly inadequate for establishing definiteness. That is, the circularity objection applies just as much to Dedekind's axiomatization of arithmetic and to the characterization of the real numbers as the unique complete ordered field. The first author explains it like this:
\begin{quote}\small
Some philosophers object that we cannot identify or secure the definiteness of our fundamental mathematical structures by means of second-order categoricity characterizations. Rather, we only do so relative to a set-theoretic background, and these backgrounds are not absolute. The proposal is that we know what we mean by the structure of the natural numbers---it is a definite structure---because Dedekind arithmetic characterizes this structure uniquely up to isomorphism. The objection is that Dedekind arithmetic relies fundamentally on the concept of arbitrary collections of numbers, a concept that is itself less secure and definite than the natural-number concept with which we are concerned. If we had doubts about the definiteness of the natural numbers, how can we assuaged by an argument relying on the comparatively indefinite concept of ``arbitrary collection''? Which collections are there? The categoricity argument takes place in a set-theoretic realm, whose own definite nature would need to be established in order to use it to establish definiteness for the natural numbers. \cite[p.~32]{Hamkins2021:Lectures-on-the-philosophy-of-mathematics}
\end{quote}
Because the second-order accounts presume a set-theoretic realm of second-order logic, the definiteness of mathematical structure that they provide is only as definite as the set-theoretic account of the second-order logic itself.

Let us conclude the paper by summarizing the main theme and points. Zermelo's quasi-categoricity results lead naturally to the question of which extensions of $\ZFC_2$ might be categorical, and we have mounted a mathematical exploration of this. The fact that all such categorical extensions of $\ZFC_2$, and even the fresh extensions, are necessarily describing comparatively small large cardinals and that the truly large large cardinals are necessarily non-categorical and nonfresh highlights the fundamental tension between categoricity and reflection in set theory. The very fact that the largest large cardinal notions are noncategorical provides reason to favor non-categoricity in our ultimate account of the set-theoretic universe, a position further strengthened by the observation that one seems inevitably pushed out of any categorical foundation by the need also to account for the categorical truth predicate on that foundation and truth-about-truth and so on into higher realms.

\printbibliography
%% Use this to get acceptable biblio output for the arXiv, which doesn't accept biblatex bbl files.
%% Must also comment out the Bibliography stuff in preamble
%\bibliographystyle{halpha}
%\bibliography{MathBiblio,HamkinsBiblio,WebPosts,PhilBiblio}

\end{document}